\documentclass[12pt]{amsart}

\usepackage[utf8]{inputenc}
\usepackage{amsmath, amsthm, amsfonts, amssymb}
\usepackage{amsaddr}
\usepackage{hyperref} 
\usepackage{todonotes}
\usepackage{xcolor}
\usepackage{float}
\usepackage[shortlabels]{enumitem}
\usepackage{tikz}
\usepackage{diagbox}

\newtheorem{theorem}{Theorem}
\newtheorem{lemma}[theorem]{Lemma}
\newtheorem{coro}[theorem]{Corollary}
\newtheorem{prop}[theorem]{Proposition}
\newtheorem*{OP}{Open problems}
\theoremstyle{definition}
\newtheorem{remark}[theorem]{Remark}

\newcommand{\eps}{\varepsilon}
\renewcommand{\phi}{\varphi}

\newcommand{\N}{\mathbb{N}}

\newcommand{\R}{\mathbb{R}}

\newcommand{\norm}[1]{\left\Vert#1\right\Vert}
\newcommand{\ipr}[1]{\left\langle#1\right\rangle}
\newcommand{\set}[1]{\Bigl\{#1\Bigr\}}

\newcommand{\be}{\begin{equation}}
\newcommand{\ee}{\end{equation}}
\newcommand{\co}{\colon\, }

\newcommand{\E}{\mathbb{E}}
\renewcommand{\P}{\mathbb{P}}

\newcommand{\calL}{\mathcal{L}}

\newcommand{\detada}{\text{\rm det}}
\newcommand{\detnon}{\text{\rm det-non}}
\newcommand{\detnonmb}{\text{\rm det-non-mb}}

\newcommand{\ranada}{\text{\rm ran}}
\newcommand{\rannon}{\text{\rm ran-non}}

\newcommand{\A}{\mathcal{A}}
\newcommand{\Adet}{\mathcal{A}_n^\detada}
\newcommand{\Adetnon}{\mathcal{A}_n^\detnon}
\newcommand{\Aran}{\mathcal{A}_n^\ranada}
\newcommand{\Arannon}{\mathcal{A}_n^\rannon}
\newcommand{\APP}{{\rm APP}}

\DeclareMathOperator{\codim}{codim}

\makeatletter
\renewcommand\subsection{\@startsection{subsection}{2}%
  \z@{-.5\linespacing\@plus-.7\linespacing}{.5\linespacing}%
  {\normalfont\bfseries}}
\makeatother

\title[Power of adaption and 
randomization]{On the power of adaption 
and randomization}
\author{David Krieg, Erich Novak, and Mario Ullrich}
\date{\today}

\begin{document}

\begin{abstract} 
We present bounds on the maximal gain of adaptive and 
randomized algorithms over non-adaptive, 
deterministic ones for 
approximating linear operators on convex sets. 
If the sets are additionally symmetric, 
then our results are optimal. 
For non-symmetric sets, we 
unify some notions of $n$-widths and s-numbers, 
and show their connection to minimal errors. 
We also discuss extensions to non-linear widths and approximation based on function values, and conclude with a list of open problems.
\end{abstract}

\maketitle

\section{Introduction and summary}

Let $X$ and $Y$ be (real or complex) Banach spaces, 
$S\in \calL(X, Y)$, i.e., a continuous linear mapping between $X$ and $Y$, 
and $F\subset X$.
We usually assume that $F$ is convex, and sometimes also that $F$ is symmetric. 
The goal is to approximate $S(f)$ for arbitrary $f\in F$ by an algorithm \mbox{$A_n\co F\to Y$} 
that has access to the values of at most $n$ linear functionals (aka measurements) applied to $f$, see Section~\ref{sec:algorithms} for precise definitions. 
Here, we ask the following question: 
\medskip

\begin{center}
How much can be gained by choosing the functionals \\ adaptively and/or randomly?
\end{center}

\medskip 

In this paper, we present several upper bounds on the largest possible gain. 
In the case that $F$ is not only convex, but also symmetric, 
we can apply known
relations of minimal worst-case errors and s-numbers as well as inequalities between different s-numbers. 
In the case that $F$ is only convex, much less is 
known and new concepts are required. 
Such non-symmetric model classes
$F$ appear quite naturally, 
for example, if the problem instances $f\in F$ 
are non-negative, monotone or convex functions. 
They may behave very differently compared to symmetric classes, as we discuss below. 
We also consider the maximal gain if only one of the two features, i.e., adaption or randomization, is allowed, and present an upper bound 
if the measurements
are given by $n$ function evaluations instead of arbitrary linear functionals.

\medskip

Let us now describe the state of the art and our main results in more detail. We start by discussing 
the power of adaption:  
How much better are algorithms 
that are allowed to choose information 
successively depending on already observed information, 
compared to those that apply the same $n$ measurements to all inputs? 
This is sometimes called the ``adaption problem''.
Note that we compare all algorithms that use the same amount of information, 
regardless of their computational cost.

\medskip

In the deterministic setting, if $F$ is additionally symmetric,
it is known that the answer is \emph{almost nothing}.
More precisely, the minimal worst-case error that can be achieved with adaptive algorithms
improves upon the one achievable 
with non-adaptive algorithms by a 
factor of at most two,
see \cite{Ba71,CW04,GM80,No96,NW1,TW80,TWW88}. 
For non-symmetric sets, it was 
proved in \cite{Novak95b}
that the largest possible gap between those errors 
is bounded above by $4 (n+1)^2$.

\medskip

For a long time,  
it was not known whether adaption helps for randomized 
algorithms if the input set $F$ is convex and symmetric. 
The problem was posed in \cite{No96} and
restated in \cite[Open Problem 20]{NW1}.
This open problem was recently solved in the affirmative by Stefan Heinrich~\cite{He24a,He24b,He24c,He24d} 
who studied (parametric) integration and approximation in 
mixed $\ell_p(\ell_q)$-spaces using standard information (function evaluations).
We stress that in this paper we mainly 
consider arbitrary linear information, 
hence the setting 
is different. 

\medskip 

For randomized algorithms using arbitrary linear information, the paper \cite{KNW24} shows that one may gain by adaption 
a factor of main order $n^{1/2}$ 
for the embedding $S\colon \ell_1^m \to \ell_2^m$
if $F$ is the unit ball of~$\ell_1^m$.
It is proved in \cite{KW24} that the same gain occurs for the embedding $S\colon \ell_2^m \to \ell_\infty^m$ and one may even gain a factor of main order $n$ for the embedding $S\colon \ell_1^m \to \ell_\infty^m$.
In these results, the dimension $m$ is chosen in (exponential) 
dependence of $n$ and hence the problem $S$ depends on $n$.
Both papers also show how one can obtain from this a single infinite-dimensional problem,
where adaption gives a speed-up of the respective main order for all $n\in \N$
by using a construction similar to the one proposed in \cite{He24d}.

\medskip

In this paper,
we give upper bounds for 
the maximal gain of randomized adaptive algorithms (the most general kind)
over
deterministic non-adaptive algorithms (the least general kind).
We denote the corresponding $n$-th minimal worst case errors 
for approximating $S$ over~$F$ 
by $e_n^\ranada(S,F)$ and $e_n^\detnon(S,F)$, see Section~\ref{sec:algorithms}. 
Our main result reads as follows;
see Theorem~\ref{thm:intro-strong} for a slightly stronger version and its proof.

\begin{theorem}\label{thm:intro}
Let $X$ and $Y$ be Banach spaces and $S\in\calL(X,Y)$. For every convex $F\subset X$ and $n\in\N$, we have
\[
e_{2n-1}^\detnon(S,F) 
\;\le\; 12\, n^{3/2}\,\bigg(\prod_{k<n} e_k^\ranada(S,F) \bigg)^{1/n}.
\]
In special cases, the following 
improvements hold:
\begin{enumerate}[a)]
    \item if $F$ is symmetric, we can replace $n^{3/2}$ with $n$,
    \item if $Y$ is a Hilbert space, we can replace $n^{3/2}$ with $n$,
    \item if $F$ is symmetric and $Y$ a Hilbert space, replace $n^{3/2}$ with $n^{1/2}$,
    \item if $X$ is a Hilbert space and $F$ its unit ball, we can replace $n^{3/2}$ with $n^{1/2}$ if we additionally replace the index $2n-1$ with $4n-1$.
\end{enumerate}
\end{theorem}

Although these bounds are of a non-asymptotic nature, see Corollary~\ref{cor:Carl}, they might be most easily understood in terms of the polynomial rate of convergence. For this, one has to realize that the geometric mean on the right hand side has the same polynomial rate of convergence as the error numbers $e_n^\ranada$. Hence, we find that adaption and randomization improve the rate of convergence by no more than 1 in the symmetric case
and by no more than $3/2$ in the non-symmetric case.
This maximal improvement is further reduced by $1/2$ 
if either the input or the target space is a Hilbert 
space.
In the case that $X$ and $Y$ are Hilbert spaces and $F$ the unit ball of $X$, then there is no gain (up to constants), see~\cite{Novak92} and Lemma~\ref{lemma92}.

\medskip

By recalling the aforementioned 
results from \cite{KNW24,KW24}, 
we see that our results for the polynomial rate 
of convergence 
are sharp in the case of symmetric classes~$F$. 
We summarize the new state of the art for the
adaption and randomization problem
in Table~\ref{table3}. 
The same results hold for the 
adaption problem in the randomized setting. 
See also Section~\ref{sec:ind} and 
Table~\ref{table:individual} for an individual 
discussion of adaption and randomization. 
For comparison, 
recall that adaption gives no speed-up 
for deterministic algorithms for all convex and 
symmetric classes $F$.

\begin{table}[ht]
\renewcommand{\arraystretch}{1.3}
\begin{center}
\begin{tabular}{|c||c|c|c|}
\hline
 \backslashbox{$Y$}{$F$} & \begin{tabular}{c} unit ball of a \\[-4pt] Hilbert space \end{tabular} & convex \& symmetric & only convex \\\hline\hline
 \begin{tabular}{c} Hilbert \\[-4pt] space \end{tabular} & no gain & 1/2 & $\le 1$ \\\hline
 \begin{tabular}{c} Banach \\[-4pt] space \end{tabular} & 1/2 & 1 & $\le 3/2$ \\\hline
 \end{tabular}
\end{center}
\medskip

\caption{Maximal gain in the rate of convergence of adaptive randomized over non-adaptive deterministic algorithms using linear information. The same table applies for the comparison of adaptive randomized with non-adaptive randomized algorithms.
}
\label{table3}
\end{table}


\medskip

A crucial tool in our analysis are 
inequalities between \emph{s-numbers}
of operators, 
see e.g.~\cite{Pie87,Pie07}, 
and between variants of those numbers for the non-symmetric case, 
see Section~\ref{sec:widths}. 


Indeed, the \emph{Gelfand numbers} $c_n$ 
characterize the error $e_n^{\detnon}$ 
of deterministic and non-adaptive algorithms up to a factor of two.
On the other hand,
it is known that the \emph{Bernstein numbers} $b_n$
are a lower bound for the error of 
deterministic and adaptive algorithms, 
see e.g.~\cite{Novak95b}.
More recently, 
based on earlier results of \cite{He92},
it has been proven in \cite{Ku16,Ku17} that also 
the error $e_n^{\ranada}$ of adaptive randomized algorithms is 
bounded below by the Bernstein numbers. 
Hence, one can obtain bounds for the ratio $e_n^{\detnon}/e_n^{\ranada}$ 
from corresponding bounds involving $c_n$ and $b_n$, 
which is the approach of this paper.


For the symmetric case, such bounds follow from the already available
estimates on the maximal difference between arbitrary s-numbers, 
see \cite{Pie87,Pie07} and the recent paper~\cite{U24}. 
For the non-symmetric case, we will use similar concepts and proof ideas. 
In particular, we will introduce the Hilbert width $h_n$ as a substitute of the Hilbert numbers, i.e., the smallest s-numbers, and prove bounds between $c_n$ and $h_n$
similar to our Theorem~\ref{thm:intro},
see Theorem~\ref{thm:widths}.

\medskip

There are many questions which remain unanswered,
even despite all the recent progress on the matter of adaption and randomization.
For instance, Table~\ref{table3} neglects any logarithmic factors 
and it is probably a very hard problem to determine the correct behavior of the maximal gain
including logarithmic factors, even in the symmetric case. 
In the non-symmetric case, we do not even know the right polynomial order of the maximal gain. 
Moreover, what is the maximal gain of non-adaptive randomized algorithms over non-adaptive deterministic algorithms?
We give a list of open problems in Section~\ref{sec:OPs}.

\medskip

Possibly the most interesting open problem is the following: 
How do the results change if we switch from algorithms that 
use arbitrary linear functionals to algorithms 
that are only allowed to use function evaluations?  
(In information-based complexity this type of information is called \emph{standard information}.)
We guess that the results 
are, under suitable conditions,  quite similar, 
but so far did not find the right ideas for a proof.  


There are results on this question, but mostly for
particular $S$ and~$F$. 
The techniques of our paper can be easily 
adapted to standard information 
in the case of 
uniform approximation on convex
subsets of $B(D)$, the space of bounded functions on a set $D$. 
That is, we consider $X=Y=B(D)$, 
equipped with the sup-norm on $D$, and 
$S= \APP_{\infty}$
being the identity on $B(D)$. 

We obtain that 
algorithms that only use function evaluations 
obey the same upper bounds as given in  Theorem~\ref{thm:intro}, see 
below and Section~\ref{sec:sampling}. 
%
Here, we only present the interesting special case that $F$ is convex and symmetric. In this case, 
it is known that we can restrict ourselves to linear algorithms, 
see~\cite{CW04} or~\cite[Thm.~4.8]{NW1}. 
Using this, we obtain bounds on the 
\emph{linear sampling numbers}. 
For $F\subset B(D)$, 
those are defined by 
\[
g_n^{\rm lin}(\APP_{\infty},F)
\;:=\; 
\inf_{\substack{x_1,\dots, x_n\in D\\ \phi_1,\dots,\phi_n\in B(D)}} \,
\sup_{f\in F} \,
\norm{f - \sum_{i=1}^n f(x_i)\, \phi_i}_{B(D)}.
\]
One might argue that 
linear sampling algorithms 
are 
the simplest  
type of algorithms, 
which are not only non-adaptive, deterministic and linear, but also only employ very restrictive (but natural) information. 

The following theorem bounds the error of linear sampling algorithms
with the error of general algorithms which 
may be nonlinear, randomized, adaptive, and based on arbitrary linear information.

\begin{theorem}\label{thm:intro-std}
Let $D$ be a set, $F$ be a convex and symmetric subset of~$B(D)$, and $\APP_{\infty}$ be the identity on $B(D)$.
Then, for all $n\in\N$, we have
\[
g_{2n-1}^{\rm lin}(\APP_{\infty},F) 
\;\le\; 6 n\,\bigg(\prod_{k<n} e_k^\ranada(\APP_{\infty},F) \bigg)^{1/n}.
\]
If $F$ is the unit ball of a Hilbert space, 
we can replace the factor $n$ with $n^{1/2}$ 
if we additionally replace the index $2n-1$ with $4n-1$.
\end{theorem}

\medskip

Theorem~\ref{thm:intro-std} is optimal in the sense that 
the factor $n$ cannot be replaced by a lower-order term.
This follows again by considering the embedding $S: \ell_1^m \to \ell_\infty^m$
as discussed in \cite{KW24}. 
See Section~\ref{sec:sampling} for some details, extensions, as well as remarks on this setting. 
Theorem~\ref{thm:intro-std} is proven by Theorem~\ref{thm:info},
a common generalization of Theorems~\ref{thm:intro} and \ref{thm:intro-std}.

\medskip


\medskip


\section{Algorithms and minimal errors}
\label{sec:algorithms}

In general, a deterministic algorithm $A_n\co F\to Y$ is an arbitrary mapping of the 
form $A_n=\phi_n\circ N_n$ 
with $N_n\co F\to \R^n$ being the \emph{information mapping}, 
and $\phi_n\co \R^n\to Y$ the \emph{reconstruction mapping}. 
We mostly pose no restriction at all on the mappings $\phi_n$ and focus on the form of $N_n$, 
see also Section~\ref{sec:OPs}. 
The most general form we consider is that an information mapping is given recursively by 
\[
N_n(f) = \left( N_{n-1}(f),\, L_n(f) \right), 
\]
where the choice of the $n$-th linear 
functional $L_n=L_n(\cdot,N_{n-1}(f))$ 
may depend on the first $n-1$ measurements. 
This is called an \emph{adaptive} choice of information, and 
we denote the collection of all such algorithms 
by $\Adet (F,Y)$, 
or just $\Adet$.


An algorithm is called \emph{non-adaptive} 
if $N_n=(L_1,\dots,L_n)$, 
i.e., the same functionals are used for every input, 
and we denote by $\Adetnon$ the corresponding class of algorithms.


Let us add that the assumption that measurements are given by linear functionals is very common in numerical analysis and approximation theory. However, also other concepts are possible. We shortly discuss this
in Section~\ref{sec:other-width}.

\medskip

For an algorithm $A_n\in\A_n^*$ with $*\in\{\detada,\detnon\}$, 
a mapping $S\co X\to Y$ and a set $F\subset X$, we define the 
\emph{worst-case error} of $A_n$ for approximating $S$ over $F$ by 
\[
e(A_n,S,F) \;:=\; \sup_{f\in F}
\|S(f)-A_n(f)\|_Y.
\]
(Note that we omit the $Y$ in $\|\cdot\|_Y$ when no confusion is possible.)

Randomized algorithms are random variables 
whose realizations are deterministic 
algorithms as described above. 

A randomized algorithm is a family of deterministic algorithms $A_n=(A_n^\omega)_{\omega\in\Omega}\subset\Adet(F,Y)$
which is indexed by a probability space $(\Omega,\mathcal{A},\P)$.
For technical reasons, 
we 
assume that the 
mapping 
$(f,\omega) \mapsto \|S(f)-A_n^\omega(f)\|_Y$
is $(\mathcal{B}_F\otimes \mathcal{A},\,\mathcal{B}_Y)$-measurable,
where $\mathcal{B}_Y$ denotes the Borel $\sigma$-algebra on $Y$,
the set $F$ is assumed to be convex, and $\mathcal{B}_F$ 
denotes the Borel $\sigma$-algebra of the topology 
associated with~$F$, 
i.e., with respect to the semi-norm 
whose unit ball is   
the convex and symmetric set $F-F$.
Then, formally, the desirable statement 
$\Adet\subset\Aran$ is not correct
since we do not assume that a deterministic algorithm 
has to be measurable. 
See \cite[Section~4.3.3]{NW1} and 
Section~\ref{sec:OPs} for a discussion 
of this technicality.

We denote the class of all such 
(possibly adaptive) algorithms by $\Aran(F,Y)$ 
and let $\Arannon(F,Y)$ 
be the class of randomized algorithms 
whose realizations are non-adaptive. 
Again, we may omit the dependence on $F$ and $Y$. 
We define the \emph{worst-case error}
of a randomized algorithm 
$A_n\in\Aran(F,Y)$ for approximating $S$ over $F$ by 
\[
e(A_n,S,F) \;:=\; \sup_{f\in F}\,
\E\, \|S(f)-A_n(f)\|_Y.
\] 

In order to compare the power of the 
just introduced types of algorithms, 
we now define the 
\emph{$n$-th minimal worst-case error} 
for approximating $S$ over $F$ by 
\[
e_n^{*}(S,F) \;:=\; \inf_{A_n\in \A_n^{*}}\,
e(A_n,S,F), 
\]
where $*\in\{\detada, \detnon, \ranada, \rannon\}$. 

\medskip

The respective concepts can 
indeed lead to very different minimal 
errors. 
Several examples, remarks and open problems 
will be presented in Section~\ref{sec:OPs}.

\section{Widths
and ${\rm s}$-numbers} 
\label{sec:widths}

\emph{Widths} 
have a long tradition in
approximation theory
and there is a whole range of widths of sets within normed spaces. 
See, e.g., Tikhomirov~\cite{Tikh60} and  Ismagilov~\cite{Ism74} 
for early treatments, 
and Pinkus~\cite{Pinkus85} 
or Lorentz et al.~\cite{LGM96} for books on the subject.
A somehow competing concept are \emph{s-numbers} of operators which play an important role in operator theory and geometry of Banach spaces, see 
Pietsch~\cite{Pie87,Pie07}. 
A short account of their 
history and potential differences 
can be found in~\cite[6.2.6]{Pie07}, see also \cite{EL13,He89}. 
Some of these widths and numbers have an obvious relation to algorithms, and hence to information-based complexity, 
while others are seemingly unrelated.
We will discuss some known relations in Section~\ref{sec:relation}.
However, we first study the relation of the relevant 
numbers
among each other. 

\medskip

We start by providing a common generalization of the above concepts. 
That is, we introduce 
various s-numbers of a mapping $S\in\mathcal{L}(X,Y)$ 
on a subset $F\subset X$. 
Alternatively, one may call them 
widths of a set $F\subset X$ with respect to a mapping $S\in\mathcal{L}(X,Y)$.

The original definitions of the corresponding widths of sets $F\subset X$ 
are obtained by considering the 
s-numbers of the identity ${\rm id}_X$ 
on $F\subset X$ (or the 
width of $F$ w.r.t.~${\rm id}_X$), 
while s-numbers of the operator $S$ are recovered by considering $F=B_X$
(or the width of $B_X$ w.r.t.\,$S$).

Here and in the following, the (closed) unit ball of $X$ is denoted by $B_X$ and 
the continuous dual space of $X$ by $X'$.

\medskip

We define the 
\emph{Gelfand numbers} of $S\in\mathcal{L}(X,Y)$ on $F\subset X$ 
by 
\[\begin{split}
c_n(S,F) 
\,&:=\,  
\inf_{L_1,\dots,L_{n}\in X'}\, 
\sup_{\substack{f,g\in F: \\ L_k(f)=L_k(g)}}
\frac12\, \bigl\|S(f)-S(g)\bigr\| \\
&=\,  
\inf_{\substack{M \subset X \text{ closed} \\ \codim(M)\le n}}\ 
\sup_{\substack{f,g\in F: \\ f-g\in M}}\,
\frac12\, \bigl\|S(f)-S(g)\bigr\|.
\end{split}\]
%
In particular, $c_0(S,F)=\frac{1}{2}{\rm diam}\bigl(S(F)\bigr)$.
Note that in the theory of s-numbers, 
there is usually an index shift 
of one and ``$s_n$'' 
is only considered for $n\ge 1$
(such that $s_1(S)=\|S\|=\frac{1}{2}{\rm diam}\bigl(S(B_X)\bigr)$). 
We use a different convention here 
because $n$ is used for 
the amount of information. 
It is well-known, and we will present the details in Section~\ref{sec:relation}, that the Gelfand numbers~$c_n$ are closely related to $e_n^{\detnon}$. 

\medskip

Other quantities that will serve as lower bounds for all minimal errors are the 
\emph{Bernstein numbers} of $S$ on $F$, 
which are defined by 
\[\begin{split}
b_n(S,F) \,:=
\sup_{\substack{\dim(V)=n+1 \\ S \text{ injective on } V}} 
\sup\big\{ &r>0 \,:\,  g+ B \subset F \text{ for some } g\in F \\[-17pt] 
&\text{and a ball } B \text{ of radius } r \text{ in } (V,\Vert\cdot\Vert_S) \big\}.
\end{split}\]
Here, 
we consider the norm on the linear space $V$ that is induced by $S$, i.e., $\Vert x \Vert_S := \Vert Sx \Vert_Y$.
If $F$ is convex and symmetric, it suffices to consider balls centered at the origin in the above definition.
We note again that these numbers coincide with the classical Bernstein widths
if $S$ is the identity on $X$.
In the special case that $F$ is a bounded subset of $X$, it is not hard to verify that we have the handy formula
\[
\begin{split}
b_n(S,F) 
\,&=\, \sup_{\substack{V\subset X \text{ affine}\\\dim(V)=n+1}}\, \sup_{g\in F\cap V} \,\inf_{f\in V\cap (X\setminus F)} \, \|S(f)-S(g)\|. 
\end{split}
\]
%

\begin{remark}
The paper \cite{Novak95b} considers instead 
the Bernstein widths of the set $S(F)$ in $Y$, i.e.,
the radius of the largest $(n+1)$-dimensional 
ball contained in $S(F)$.
This coincides with the Bernstein 
numbers of~$S$ on~$F$
as defined above
if $S$ is injective. If $S$ is not injective, 
the widths of the set $S(F)$ may be larger. 
\end{remark}

\medskip

We will see in Section~\ref{sec:relation}, that one obtains bounds on 
$e_n^{\rm \detnon}/e_n^{\ranada}$ 
from corresponding bounds involving $c_n$ and $b_n$, which is the approach of this paper.
For this, we want to employ proof ideas that have already been used for bounding the maximal difference between s-numbers, 
see e.g.~\cite{Pie87,Pie07}. 
Inspired by Hilbert numbers, see~\cite{Bauhardt}, which are the smallest s-numbers, we introduce the 
\emph{Hilbert numbers} of $S\in\mathcal{L}(X,Y)$ on $F\subset X$
by 
\[
\begin{split}
h_n(S,F) \,:=\, \sup\Big\{&c_n\bigl(BSA, B_{\ell_2}\bigr)\co B\in\calL(Y,\ell_2) \text{ with } \Vert B\Vert \le 1,\\ 
 &A\in \mathcal{L}(\ell_2,X) \text{ and } x\in F \text{ with } A(B_{\ell_2}) + x \subset F \Big\}.
%
\end{split}
\]
In this definition, we can replace $c_n$ with $b_n$ since both numbers coincide for operators $T\in \mathcal{L}(\ell_2,\ell_2)$; they are both equal to the $(n+1)$-st singular value of $T$,
see, e.g., \cite{Pietsch-s}.

\medskip

One of the key ingredients to our results will be bounds between these numbers. 
First, note that they are related in the same way as 
the corresponding s-numbers, 
see~\cite{Bauhardt,Pietsch-s}. 

\begin{prop} \label{prop1}
Let $X$ and $Y$ be Banach spaces and $S\in\calL(X,Y)$. 
For every convex $F\subset X$ and $n\in\N_0$, 
we have 
\[
h_n(S,F) \;\le\; b_n(S,F) \;\le\; c_n(S,F). 
\]
Equalities hold if $X$ and $Y$ are Hilbert spaces and $F=B_X$.
\end{prop}

\begin{proof}[Proof of Proposition~\ref{prop1}]
In order to prove $b_n \ge h_n$,
let $B\in\calL(Y,\ell_2)$ with $\Vert B \Vert \le 1$ as well as $A\in\calL(\ell_2,X)$ and $g\in X$ with $A(B_{\ell_2}) + g \subset F$.
For any $\beta < b_n\bigl(BSA, B_{\ell_2}\bigr)$,
there exists an $(n+1)$-dimensional linear space $V \subset \ell_2$ such that $BSA$ is injective on $V$
and for any $v\in V$ it holds that $\Vert BSAv \Vert_2 \le \beta$ implies $v\in B_{\ell_2}$.
First, we observe that the injectivity of $BSA$ implies that $S$ is injective on $W=A(V)$ and that $W$ is $(n+1)$-dimensional.
Moreover, let $f\in W$ with $\Vert Sf\Vert_Y \le \beta$.
Choose $v\in V$ with $f=Av$.
We have
\[
 \Vert BSAv\Vert_2 \le \Vert SAv \Vert_Y = \Vert Sf \Vert_Y \le \beta
\]
and hence $v\in B_{\ell_2}$.
By assumption, this implies $f+g\in F$ for all such $f$.
Hence, $F$ contains an $\Vert \cdot \Vert_S$-ball of radius $\beta$ in $W$ and we have
$b_n(S,F) \ge \beta$.
    Taking the supremum over all $\beta$ gives
    \[
     b_n(S,F) \,\ge\, b_n(BSA, B_{\ell_2})
    \]
and taking the supremum over all $B$, $A$ and $g$ as above gives
    \[
    b_n(S,F) \,\ge\, h_n(S,F).
    \]
    
In order to show $c_n \ge b_n$,
let $\beta<b_n(S,F)$ be arbitrary 
and let $V\subset X$ be an $(n+1)$-dimensional subspace 
such that $S$ is injective on $V$ 
as well as $m\in F$ such that
$h\in V$ and $\Vert Sh \Vert_Y \le \beta$ imply $m+h \in F$. 
Now, for all $L_1,\dots,L_n\in X'$, 
there must be some $h\in V\setminus\{0\}$ with $L_i(h)=0$ for all $i\le n$.
We choose $h$ such that $\Vert Sh \Vert_Y=\beta$, which implies $f=m+h \in F$ and $g=m-h \in F$.
Note that $L_i(f)=L_i(g)$ for all $i\le n$.
Moreover, $\frac12 \Vert Sf-Sg \Vert = \beta$ and hence $c_n(S,F)\ge \beta$.\\
\end{proof}
\medskip

Here, we prove the
following reverse inequalities,  
which are reminiscent to the corresponding bounds for s-numbers, 
see Remark~\ref{rem:known}.

\medskip

\begin{theorem}\label{thm:widths}
Let $X$ and $Y$ be Banach spaces and $S\in\calL(X,Y)$. 
For every convex $F\subset X$ and 
$n\in \N$, 
we have
\[
c_{n-1}(S,F) 
\;\le\;
\left(\prod_{k=0}^{n-1} c_k(S,F)\right)^{1/n}
\;\le\; n^{3/2}\,\left(\prod_{k=0}^{n-1} h_k(S,F)\right)^{1/n}.
\]
In special cases, the following improvements hold: 
\begin{enumerate}[a)]
    \item if $F$ is symmetric, we can replace the exponent $3/2$ with $1$,
    \item if $Y$ is a Hilbert space, we can replace the exponent $3/2$ with $1$,
    \item if $F$ is symmetric and $Y$ a Hilbert space, replace $3/2$ with $1/2$,
    \item if $X$ is a Hilbert space and $F$ its unit ball, we can replace the exponent $3/2$ with $1/2$ if we also replace all $c_{k}$ with $c_{2k}$. 
\end{enumerate}
\end{theorem}

\medskip

\begin{proof}[Proof of Theorem~\ref{thm:widths}] \label{proof:widths}
Let $S\in\mathcal{L}(X,Y)$ and $F\subset X$ be convex.

\medskip

\underline{General case}:
We first show that, for fixed $\eps>0$, we can find 
$f_0,g_0$, $f_1,g_1,\ldots \in F$ and 
$L_0,L_1,\ldots \in {X'}$ 
such that, with $p_k:=\frac{f_k-g_k}{2}$, we have $L_j(p_k)=0$ for $j<k$ and 
$(1+\eps) L_k(p_k)>c_k(S,F)$ for $k=0,1,\ldots$. 
See~\cite[p.~132]{Novak95b} for a similar proof. 

The proof is by induction. 
Let $k\in\N_0$ and assume that $f_j,g_j$ and $L_j$ for $j<k$ are already found. 
Define 
\[
M_k \,:=\, \set{p\in X\co L_j(p)=0 \;\;\text{ for }\; j< k}. 
\] 
Since $\codim{M_k}\le k$, we can choose $f_k,g_k\in F$ 
with $p_k\in M_k$ and 
\begin{equation}\label{eq:choosep}
(1+\eps)\,\|Sp_k\| \,\ge\, c_k(S,F).  
\end{equation}
By the Hahn-Banach theorem, there is $\lambda_k\in B_{Y'}$ with $\lambda_k(Sp_k)=\|Sp_k\|$
and hence 
\begin{equation} \label{proof1}
    \lambda_k(Sp_k) \;\ge\; (1+\eps)^{-1} \, c_k(S,F).
\end{equation}
We finish the induction step by setting $L_k=\lambda_k\circ S\in X'$.

\medskip

For $n\in\N$, 
we now define $g=\frac{1}{n}\sum_{i<n} \frac{f_i+g_i}{2} \in F$ 
and the operators 
\begin{equation}\label{eq:opA}
A(\xi) \,:=\, \frac{1}{n}\sum_{i<n} \xi_i p_i \in X, 
\quad \xi=(\xi_i)_{i<n} \in\ell_2^n,
\end{equation}
and
\[
B(y) \,:=\, \frac{1}{\sqrt{n}} \bigl({\lambda_i}(y)\bigr)_{i<n}\in\ell_2^n, \qquad y\in Y, 
\]
and consider the mapping $S_n:=BSA$.
We observe that $\|B\|\le 1$ and, for all $\xi\in[-1,1]^n$, due to convexity, it holds
\[A(\xi) + g \,=\, \frac{1}{n} \sum_{i<n} \left( \frac{1+\xi_i}{2} f_i + \frac{1-\xi_i}{2} g_i \right) \in F.\]
In particular, $A(B_{\ell_2^n})+g \subset F$.
This gives, for any $k< n$, 
\[
 c_k(S_n,B_{\ell_2^n})\le h_k(S,F).
\]
%

\medskip

Our bounds are obtained by considering the determinant of $S_n\co\ell_2^n\to\ell_2^n$.
Since $S_n$  
is generated by the triangular matrix $n^{-3/2} (L_j(p_i))_{i,j<n}$, 
we have 
\[
\det(S_n) \,\ge\, 
\prod_{k<n} \frac{c_k(S,F)}{n^{3/2}(1+\eps)}.
\]
On the other hand, the determinant is multiplicative and equals the product of the singular values, which in turn equal the Gelfand widths $c_k(S_n,B_{\ell_2^n})$. 
Therefore, 
we obtain 
\[
\prod_{k<n} \frac{c_k(S,F)}{n^{3/2}(1+\eps)}
\;\le\;\det(S_n) 
\;=\;  \prod_{k<n} c_k(S_n,B_{\ell_2}) 
\;\le\;  \prod_{k<n} h_k(S,F).
\]
Taking the infimum over all $\eps>0$ and using 
$c_k(S,F) \ge c_{n-1}(S,F)$ for $k<n$,
we obtain 
\[
c_{n-1}(S,F) 
\;\le\; n^{3/2}\,\bigg(\prod_{k<n} h_k(S,F) \bigg)^{1/n}.
\]



\underline{$F$ symmetric}:
If $F$ is additionally symmetric, 
then one has $p_i=\frac{f_i-g_i}{2} \in F$
such that we can redefine $A$ 
by $A(\xi):= \frac{1}{\sqrt{n}}\sum_{i<n} \xi_i p_i$
to have $A(B_{\ell_2}) \subset F$.
Hence, 
we can continue with the triangular matrix 
$S_n=n^{-1}(L_j(p_i))_{i,j<n}$ 
to obtain the improved bound.

\medskip

\underline{$Y$ Hilbert space}:
If $Y$ is a Hilbert space,
we can choose the functionals $L_k$ from the Hahn-Banach theorem explicitly as 
\begin{equation*}\label{eq:infohilbert}
L_k\,:=\, \bigg\langle\,\cdot\ ,\, \frac{Sp_k}{\Vert Sp_k\Vert_Y} \bigg\rangle,
\end{equation*}
where $\ipr{\cdot,\cdot}$ is the inner product in $Y$. 
Hence, by the definition of the sets $M_k$,
we see that the $Sp_k$ for $k\le n$ are pairwise orthogonal.
We can hence skip the factor $n^{-1/2}$ in the definition of $B$ 
and put 
$B(y):=(\lambda_i(y))_{i<n}$ 
while preserving the property $\Vert B \Vert \le 1$. 
Thus, also in this case, we can continue with the triangular matrix 
$S_n=n^{-1}(L_j(p_i))_{i,j<n}$. 

\medskip

\underline{F symmetric, $Y$ Hilbert space}:
We combine the modifications from the previous two cases
and work with the matrix $S_n=n^{-1/2}(L_j(p_i))_{i,j<n}$.




\medskip

\underline{$F$ unit ball of a Hilbert space}: 
The proof is similar to the general case.
We show by induction that for fixed $\eps>0$, 
there are orthogonal vectors
$p_0,p_1,\ldots\in B_X$ and
$L_0,L_1,\ldots\in X'$ 
such that $L_j(p_k)=0$ for $j<k$ and 
$(1+\eps) L_k(p_k)\ge c_{2k}(S,B_X)$ for $k=1,2,\ldots$

Assume that $p_j,L_j$ for $j<k$ are already found, 
and define 
\[
M_n \,:=\, \set{p\in X\co L_j(p)=0 \text{ and } \langle p_k,p \rangle = 0 \text{ for } j< k}. 
\] 
Since $\codim{M_k}\le 2k$, we can choose $p_k\in B_X$ 
with $p_k\in M_k$ and 
\[
(1+\eps) \|Sp_k\| \,\ge\, c_{2k}(S,B_X).  
\]
By the Hahn-Banach theorem, there is 
$\lambda_k\in B_{Y'}$
with $\lambda_k(Sp_k)=\|Sp_k\|$. 

\medskip

For $n\in\N$, we define the operators $A \in \mathcal{L}(\ell_2^n,X)$ and $B\in \mathcal{L}(Y,\ell_2^n)$ by
\[
A(\xi) \,:=\, \sum_{i<n} \xi_i p_i, \qquad 
B(y) \,:=\, \frac{1}{\sqrt{n}}\bigl(\lambda_i(y)\bigr)_{i<n}
\]
such that $\Vert A \Vert \le 1$ and $\|B\|\le 1$.
The mapping $S_n:=BSA$
is generated by the triangular matrix $n^{-1/2}(L_j(p_i))_{i,j<n}$ , 
where $L_j:=\lambda_j\circ S$. 
This gives
\[
\prod_{k<n} \frac{c_{2k}(S,B_X)}{n^{1/2}(1+\eps)}
\;\le\;\det(S_n) 
\;=\;  \prod_{k<n} c_k(S_n,B_{\ell_2^n}) 
\;\le\;  \prod_{k<n} h_k(S,B_X). 
\]
Taking the infimum over $\eps>0$ and using $c_{2k} \le c_{2n-2}$ for $k<n$, we get
\[
c_{2n-2}(S,F) 
\;\le\; n^{1/2}\,\bigg(\prod_{k<n} h_k(S,F) \bigg)^{1/n}.
\]
\end{proof}

\begin{remark}\label{rem:known}
a) Note the ``oversampling'' in Theorem~\ref{thm:widths} in the case that the input space is a Hilbert space, 
where we consider $c_{2n-2}$ on the left hand side. 
We do not know if this is necessary.

b) We mentioned above that the Gelfand and Hilbert numbers 
coincide with the corresponding s-numbers if $F$ is the unit ball of $X$. 
In this case, Theorem~\ref{thm:widths} was known, see~\cite[2.10.7]{Pie87} and~\cite[Theorem 11.12.3]{Pietsch-ideals} or~\cite{U24} for a streamlined presentation.
\end{remark}

\medskip

It may be desirable to compare $c_n$ directly with $h_n$ instead of the geometric mean of $h_0,\ldots,h_n$.
Under the regularity condition that $h_k/h_{2k}$ is bounded,
such a comparison is obtained from the following lemma.

\begin{lemma}\label{lem:regular}
Let $n\in\N$ be even and $z_1 \ge \ldots \ge z_n>0$. 
Moreover, let $c>0$ with $z_k \le c\, z_{2k}$ for all $k\le n/2$.
Then
\[
\bigg(\prod_{k=1}^n z_k \bigg)^{1/n} \;\le\; c^4\,z_n.
\]
\end{lemma}

\medskip 

\begin{proof}[Proof of Lemma~\ref{lem:regular}]
Choose $v\in\N_0$ such that $n/2 < 2^v \le n$.
For $2^j \le k \le n$, we get 
\[z_k \le z_{2^j} \le c^{v-j} z_{2^v}
\le c^{v-j} z_{n/2}
\le c^{v-j+1} z_n\]
so that
\[
\bigg(\prod_{k=1}^n z_k \bigg)^{1/n} \le 
\bigg( \prod_{j=0}^{v} \prod_{2^j \le k < 2^{j +1}} c^{v-j+1}  \bigg)^{1/n} \cdot z_n
= c^\kappa \, z_n
\]
with
\[
 \kappa = \frac1n \sum_{j=0}^{v} (v-j+1) 2^j  
 = \frac{2^{v+1}}{n} \sum_{i=1}^{v+1} i 2^{-i} \le 4.
\]
\end{proof}

 Lemma~\ref{lem:regular} is a fitting estimate for sequences of polynomial decay: 
 If $z_n=n^{-\alpha}$, then $c=2^\alpha$ 
 is independent of $n$. 
 For sequences of super-polynomial decay, it might be better to use the simple estimate
 \begin{equation}\label{eq:exponential}
 \bigg(\prod_{k=1}^n z_k \bigg)^{1/n} \;\le\; \sqrt{\phantom{l} z_1 \cdot z_{n/2}}.
 \end{equation}

\medskip

In the case that $Y$ is a Hilbert space,
there also is the following alternative bound which works for individual $n$
without any regularity condition as in Lemma~\ref{lem:regular}.
On the downside, the upper bound is in terms of the (possibly larger)
Bernstein numbers instead of the Hilbert numbers.

\begin{theorem}\label{thm:hilbert}
Let $X$ be a Banach space, $H$ be a Hilbert 
space and $S\in\calL(X,H)$. 
For every convex $F\subset X$ and $n\in\N_0$, 
we have
\[
c_n(S,F) 
\;\le\; (n+1)\cdot b_n(S,F).
\]  
We can replace $(n+1)$ by $\sqrt{n+1}$ 
if $F$ is additionally symmetric.
\end{theorem}

\medskip

\begin{proof}[Proof of Theorem~\ref{thm:hilbert}]
 We take $g_k$, $f_k$ and $p_k=\frac{f_k-g_k}{2}$ from the proof of Theorem~\ref{thm:widths} (general case),
 and put $r= \frac{c_{n-1}}{1+\eps}$. 
 The $n$-dimensional space $V$ spanned by $p_0,\ldots,p_{n-1}$ with
 the norm $\Vert\cdot \Vert_S$ is a Hilbert space.
 The vectors $p_k$ have norm at least $r$ and,
 as observed in 
 the proof of Theorem~\ref{thm:widths},
 they are orthogonal in $V$.
 If $F$ is convex and symmetric,
 we have $\pm p_k \in F$
 and so $F$ contains  
 a ball of radius $\frac{r}{\sqrt{n}}$ in $V$. 
 This proves the claim, i.e., $b_{n-1}(S,F) \ge \frac{c_{n-1}(S,F)}{\sqrt{n}}$.
 \medskip

 In the non-symmetric case,
 we already observed that $A(B_{\ell_2^n})+g \subset F$
 with $A$ and $g$ as in \eqref{eq:opA},
 so that $F$ contains a ball of radius $\frac{r}{n}$.\\
\end{proof}

The paper \cite{Pu79} contains bounds similar to Theorem~\ref{thm:hilbert} with the Gelfand widths replaced by the Kolmogorov widths. Since the target space is a Hilbert space,
the Kolmogorov widths are larger than
the Gelfand widths, 
see, e.g., \cite[Prop. 5.2]{Pinkus85} and 
Section~\ref{sec:OPs}.
This means that the bounds of Theorem~\ref{thm:hilbert} are 
known up to constants.
We presented the proof anyway since the 
result follows with little effort from 
our other observations.

\medskip

Due to their relations with the previously defined minimal worst-case errors (as discussed in the next section),
the Gelfand, Bernstein and Hilbert numbers as considered above are of particular interest to us.
Nonetheless, in Section~\ref{sec:OPs}, we will mention some 
other types of widths that may be of independent interest
and discuss how 
Theorem~\ref{thm:widths} applies to these widths.

\medskip

\section{Widths versus minimal errors}
\label{sec:relation}

In this section, we discuss how the Gelfand, Bernstein and Hilbert numbers are related to minimal errors
and hence obtain bounds between the different types of minimal errors.

\medskip


First, note that the Gelfand numbers characterize
the minimal worst-case error of deterministic algorithms up to a factor of two.
This is a special case of a classical result in information-based complexity,
see, e.g, \cite[Sec.~4.1]{NW1}.
In our setting, the result reads as follows.

\medskip

\begin{prop}\label{prop:upper-c}
    Let $X$ and $Y$ be Banach spaces and $S\in\calL(X,Y)$. 
    For every 
    $F\subset X$ and $n\in\N_0$, we have
    \[
     c_n(S,F) \,\le\, e_n^{\detnon}(S,F) \,\le\, 2\, c_n(S,F).
    \]
    If $F$ is convex and symmetric then 
    \[
     c_n(S,F) \,\le\, e_n^{\det} 
     (S,F) \, \le \, e_n^{\detnon}(S,F) \,\le\, 2\, c_n(S,F).
    \]
\end{prop}

\medskip

We turn to the relation of Bernstein 
numbers and minimal errors. 
It is known, 
see~\cite{Novak95b}, that $b_n(S,F)$ 
may serve as a lower bound for 
the error of adaptive deterministic algorithms.

\medskip 

\begin{prop}\label{prop:upper-ada}
    Let $X$ and $Y$ be Banach spaces and 
    $S\in\calL(X,Y)$. 
    For every 
    $F\subset X$ and $n\in\N_0$, we have
    \[
   e_n^{\det}(S,F) \,  \ge \,    b_n(S,F) .
    \]
\end{prop}

\medskip

\begin{proof}[Proof of Proposition~\ref{prop:upper-c} 
and \ref{prop:upper-ada}]
The technique of the proof is the same 
for both results and 
well known 
(but a factor of 2 is missing 
in Proposition 1 of \cite{Novak95b})
and hence we concentrate on 
Proposition~\ref{prop:upper-ada}. 
Let $A_n = \phi_n \circ N_n$
be an algorithm based on the information 
$N_n : F \to \R^n$ that might be adaptive. 
We fix the non-adaptive and linear mapping
$N_n^* = (L_1^*, \dots , L_n^*) : 
F \to \R^n$ that is taken for the midpoint 
$g$ of a ball $g + B \subset F$. 
The mapping $N_n^*$ cannot be injective 
and there exists a point $\tilde f$ on the sphere 
of $B$
with 
$N^*(\tilde f+ g) = N^* (g) = N^*(g-\tilde f)$,
hence
$N(g+\tilde f) = N(g-\tilde f)$.
Then $A_n$ cannot distinguish between the two 
inputs and we obtain the lower bound. \\ 
\end{proof}
 
In the case that $X$ and $Y$ are Hilbert spaces 
and $F$ is the unit ball of $X$,
it is shown in \cite{Novak92} that the 
Bernstein numbers also yield lower bounds for 
randomized adaptive algorithms.
Here we use a slightly different error 
criterion and hence formulate a lemma. 

\begin{lemma}  \label{lemma92} 
Let $H$ and $G$ be Hilbert spaces 
and $S\in\calL(H,G)$. 
For every $n\in\N_0$, we have 
\begin{equation}  \label{eq:Novak92}
e_n^{\ranada}(S,B_H) \,\ge\, \frac{1}{2}\, 
b_{2n-1}(S,B_H)
\,=\, \frac{1}{2}\, e_{2n-1}^{\detnon}(S,F).
\end{equation}
\end{lemma}

\medskip

With a different constant, Lemma~\ref{lemma92} is implied by \cite[Cor.\,2]{He92}.


\begin{proof}[Proof of Lemma~\ref{lemma92}]

Let $0<b<b_{2n-1}(S,F)$. There exists a subspace $V \subset H$ with dimension $2n$ such that $\Vert Sf \Vert_G \ge b \Vert f \Vert_H$ for all $f\in V$. Let $W:=S(V)$.
We choose $R\in\mathcal{L}(\R^{2n},V)$ and $Q\in\mathcal{L}(W,\R^{2n})$, where $\R^{2n}$ is considered with the Euclidean norm, such that $\Vert R \Vert \le 1$ and $\Vert Q \Vert \le b^{-1}$ and 
such that $QSR$ equals the identity ${\rm id}_{2n}$ on $\R^{2n}$.
If $A_n$ is a randomized algorithm for $S$ with error less than $b/2$,
then $QA_nR$ is a randomized algorithm for ${\rm id}_{2n}$ with error less than $1/2$.
It hence suffices to prove the lower bound $1/2$ in the case $S={\rm id}_{2n}$.


We let $P_{2n}$ be the uniform 
distribution on the sphere 
of $\R^{2n}$.
An application of Fubini's theorem
(known as Bakhvalov's proof technique, see~\cite{Ba15} or \cite[Section~4.3.3]{NW1})
gives
\[
 e_n^\ranada(S,F) \,\ge\, \inf_{A_n}\, 
 \int \Vert f - A_n(f) \Vert \, {\rm d} P_{2n}(f),
\]
where the infimum runs over all deterministic and measurable algorithms $A_n\in \mathcal{A}_n^\detada$.
Hence let $A_n = \phi \circ N_n$ be a measurable deterministic algorithm with adaptively chosen information 
$N_n=(L_1, \dots , L_n)$ and let $f$ be distributed according to $P_{2n}$. 
Assume that the functionals $L_i$ are chosen orthonormal; 
this is no restriction. 
For each $y$ in the unit ball of $\R^n$, the information $N_n(f)=y$ defines a sphere $\mathbb{S}_y$ of radius $r_y=\sqrt{1-\Vert y\Vert^2}$.
We have
\[
 \int \Vert f - A_n(f) \Vert \, {\rm d} P_{2n}(f) 
\,=\, \int \int \Vert f - \phi(y) \Vert \, {\rm d} \mu_y(f) \, {\rm d} \nu(y)  
\]
where $\mu_y$ is the uniform distribution on $\mathbb{S}_y$ and $\nu$ is the distribution of $N_n(f)$.
The inner integral is minimized if $\phi(y)$ equals the center of $\mathbb{S}_y$, so that we have
\[
 \int \Vert f - A_n(f) \Vert \, {\rm d} P_{2n}(f) 
\,\ge\, \int r_y \, {\rm d} \nu(y)  
\,\ge\, \int r_y^2 \, {\rm d} \nu(y).  
\]
From the symmetry of $P_{2n}$ it follows that 
$\nu$ does not depend on $N_n$. 
We choose $N_n(f)=(f_1,\hdots,f_n)$ and get
\[
 \int \Vert f - A_n(f) \Vert \, {\rm d} P_{2n}(f) 
\,\ge\, \int \sum_{i=n+1}^{2n} f_i^2 \, {\rm d} P_{2n}(f)
\,=\, \frac{1}{2}.
\]
The last identity holds since the $f_i^2$ are identically distributed so that their expected value equals $1/(2n)$. 
See Proposition~\ref{prop1} for the equality $b_n=e_n^{\detnon}$.\\
\end{proof}

More recently, it has been shown in \cite{Ku16,Ku17} 
that also for Banach spaces $X$ and $Y$, 
it holds that
\begin{equation}\label{eq:Kunsch}
e_n^{\ranada}(S,F) \,\ge\, \frac{1}{30}\, 
b_{2n-1}(S,F).
\end{equation} 
The result of \cite{Ku17} is proven only in the symmetric case but it remains valid if $F$ is only convex.
On the other hand, the result \eqref{eq:Novak92} for the Hilbert case easily implies the following.

\begin{prop}\label{prop:lower-h}
    Let $X$ and $Y$ be Banach spaces and $S\in\calL(X,Y)$. For every convex $F\subset X$ and $n\in\N$, we have
    \[
     e_n^{\ranada}(S,F) \,\ge\, \frac{1}{2}\, h_{2n-1}(S,F).
    \]
\end{prop}

\medskip

\begin{proof}[Proof of Proposition~\ref{prop:lower-h}]
    Let $A_n \in \mathcal{A}_n^\ranada(X,Y)$
    and let $B\in\calL(Y,\ell_2)$ with $\Vert B \Vert \le 1$ as well as $A\in\calL(\ell_2,X)$ and $g\in F$ with $A(B_{\ell_2}) + g \subset F$.
    Then we have $BA_nA \in \mathcal{A}_n^\ranada(\ell_2,\ell_2)$. This algorithm uses information of the form $L_k'=L_k A\in \ell_2'$, if $L_k\in X'$ is the information used by $A_n$.
    Note that $A$ is continuous in the norm induced by $F$ due to $A(B_{\ell_2}) \subset F-F$ and hence the algorithm is measurable.
    By \eqref{eq:Novak92}, 
    we have
    \[
    e_n^\ranada(BA_nA,BSA,B_{\ell_2})
    \,\ge\, \frac12\, b_{2n-1}(BSA,B_{\ell_2}).
    \]
    On the other hand,
    \begin{multline*}
    e_n^\ranada(BA_nA,BSA,B_{\ell_2})
    \,=\, e_n^\ranada(BA_n(A+g),BS(A+g),B_{\ell_2})\\
    \,\le\, e_n^\ranada(BA_n,BS,F)
    \,\le\, e_n^\ranada(A_n,S,F),
    \end{multline*}
    where we used $(A+g)(B_{\ell_2}) \subset F$ in the first and $\Vert B \Vert \le 1$ in the second inequality.
    So,
    \[
    e_n^\ranada(A_n,S,F)
    \,\ge\, \frac12\, b_{2n-1}(BSA,B_{\ell_2}).
    \]
    Taking the supremum over all $B$, $A$ and $g$ as above gives the result.\\
\end{proof}

We point out that the Hilbert numbers can be much smaller than the Bernstein numbers.
For example, if $S$ is the identity on $\ell_1$ and $F$ the unit ball of $\ell_1$, 
then the Bernstein numbers are equal to one, see \cite{Pietsch-s}, 
while the Hilbert numbers are of order $n^{-1/2}$, see \cite[2.9.19]{Pie87}.
So in general, one should prefer the bound \eqref{eq:Kunsch} over Proposition~\ref{prop:lower-h}.
However, since our upper bounds are in terms of the Hilbert numbers anyway,
we will obtain a better constant in the overall comparison if we work with Proposition~\ref{prop:lower-h} instead of \eqref{eq:Kunsch}.

\medskip

\section{The main result}

We now arrive at our main result, Theorem~\ref{thm:intro},
which we present here in a slightly stronger form.
\begin{theorem}\label{thm:intro-strong}
Let $X$ and $Y$ be Banach spaces and $S\in\calL(X,Y)$. For every convex $F\subset X$ and $n\in\N$, we have
\[
 \left(\prod_{k<2n} e_k^{\detnon}(S,F)\right)^{1/(2n)}
    \,\le\, 2^{7/2}\, n^{3/2} \, \left(\prod_{k<n} e_k^{\ranada}(S,F)\right)^{1/n}.
\]
In special cases, the following 
improvements hold:
\begin{enumerate}[a)]
    \item if $F$ is symmetric, we can replace $n^{3/2}$ with $n$,
    \item if $Y$ is a Hilbert space, we can replace $n^{3/2}$ with $n$,
    \item if $F$ is symmetric and $Y$ a Hilbert space, replace $n^{3/2}$ with $n^{1/2}$,
    \item if $X$ is a Hilbert space and $F$ its unit ball, we can replace $n^{3/2}$ with $n^{1/2}$ if we also replace the range $k<2n$ with $k<4n$.
\end{enumerate}
\end{theorem}

\begin{proof}[Proof of Theorems~\ref{thm:intro} and~\ref{thm:intro-strong}]
    A successive application of Proposition~\ref{prop:upper-c},
    Theorem~\ref{thm:widths},
    the monotonicity of the Hilbert widths, and Proposition~\ref{prop:lower-h} (in the weaker form $h_{2k} \le 2 e_k^\ranada$ for all $k\in\N_0$) gives
    \begin{align*}
        \prod_{k<2n} &e_k^{\detnon}(S,F)
        \,\le\, 2^{2n} \cdot \prod_{k<2n} c_k(S,F)
        \,\le\, 2^{5n} n^{3n} \cdot \prod_{k<2n} h_k(S,F)\\
        &\le\, 2^{5n} n^{3n} \cdot \prod_{k<n} h_{2k}(S,F)^2
        \,\le\, 2^{7n} n^{3n} \cdot \prod_{k<n} e_k^{\ranada}(S,F)^2.
    \end{align*}
    The modifications in the special cases are obvious.\\
\end{proof}



Since 
estimates in terms of geometric means might be unfamiliar to the reader,
we 
present 
a corollary of Theorem~\ref{thm:intro}
which is reminiscent of Carl's inequality.

\begin{coro}\label{cor:Carl}
Let $X$ and $Y$ be Banach spaces and $S\in\mathcal{L}(X,Y)$.
For every convex 
    $F\subset X$, $n\in\N$ and $\alpha>0$, we have
\vspace*{5pt}
\[
e_{2n-1}^\detnon(S,F) 
\;\le\; C_\alpha \, n^{-\alpha+3/2}\cdot \sup_{k< n}\left( (k+1)^{\alpha}\, e_k^\ranada(S,F)\right),
\]
where $C_\alpha \le 12^{\alpha+1}$.
In accordance with the special cases given in Theorem~\ref{thm:intro}, 
the exponent $3/2$ can be replaced with $1$ or $1/2$.
\end{coro}

\medskip

\begin{proof}[Proof of Corollary~\ref{cor:Carl}]
    If $K$ denotes the supremum on the right hand side,
    then $e_k^\ranada(S,F)\le K (k+1)^{-\alpha}$ for all $k<n$.
    Now the statement follows from Theorem~\ref{thm:intro}
    and Lemma~\ref{lem:regular} with $z_n=K n^{-\alpha}$.\\
\end{proof}

We also write explicitly the implication for the polynomial rate of convergence,
which has been presented in the introduction as Table~\ref{table3}.
The polynomial rate of convergence of a sequence $(z_n) \subset [0,\infty)$
is defined by 
\begin{equation} \label{eq:rate}
{\rm rate}(z_n)
 \,:=\, \sup\Big\{ \alpha>0 \ \Big|\ 
 \exists C\ge 0: \forall n\in\N: z_n \le C n^{-\alpha}
 \Big\}.
\end{equation}
We only give the result for the symmetric case, where the bounds are sharp up to logarithmic factors. 
It should be obvious enough what the corresponding results in the non-symmetric case look like.

\medskip

\begin{coro}\label{coro:rate}
Let $X$ and $Y$ be Banach spaces and $S\in\mathcal{L}(X,Y)$. 
For every convex and symmetric
    $F\subset X$, we have
\vspace*{5pt}
\[
{\rm rate}\Big(e_n^{\detnon}(S,F)\Big) 
\;\ge\; {\rm rate}\Big(e_n^{\ranada}(S,F)\Big) \,-\, 1.
\]
If either $Y$ is a Hilbert space or $F$ is the unit ball of a Hilbert space $X$, we even have
\[
{\rm rate}\Big(e_n^{\detnon}(S,F)\Big) 
\;\ge\; {\rm rate}\Big(e_n^{\ranada}(S,F)\Big) \,-\, 1/2.
\]
Moreover, in each of these cases,
equality can occur.
\end{coro}

\begin{proof}[Proof of Corollary~\ref{coro:rate}]
    The inequalities are implied by Corollary~\ref{cor:Carl}.
    Equality occurs for the examples from \cite[Cor.\,4.2]{KW24}, 
    \cite[Rem.\,3.4]{KNW24}, and \cite[Rem.\,4.3]{KW24},
    respectively.\\
\end{proof}

%


\goodbreak
\medskip

\section{Examples and related problems}
\label{sec:OPs}

In this section we give further details and extensions of our result and discuss related problems. 
In particular, we 
analyze the individual influence
of adaption and randomization
on the minimal worst-case error, 
and present 
those examples which exhibit the largest gain known to us.
In addition, 
we show how our results apply to 
other \emph{non-linear widths} and to 
approximation based on standard information, 
i.e., function evaluations, 
as shown in Theorem~\ref{thm:intro-std}.
We also present a list of open problems.

\subsection{The individual power of adaption and randomization}
\label{sec:ind}

By the results of the previous sections, 
we know how much adaption and randomization 
can help if they are allowed 
\emph{together}. 
That is, we have a good understanding of the maximal gain from $\Adetnon$ to $\Aran$.
In the symmetric case, 
we even know that 
our bounds are optimal up to logarithmic factors. 
However, our knowledge about the individual 
power of randomization or adaption
still has several gaps.
\medskip

Let us first talk about upper bounds.
Intuitively, it is clear that the gain of randomization or adaption alone cannot be larger than the gain of adaption and randomization together.
Let us make this a corollary.

\begin{coro}\label{coro:individual}
Let $X$ and $Y$ be Banach spaces and $S\in\calL(X,Y)$. For every convex and bounded $F\subset X$ and $n\in\N$, we have
\[
e_{2n-1}^*(S,F) 
\;\le\; C n^{3/2}\,\bigg(\prod_{k<n} e_k^\square(S,F) \bigg)^{1/n},
\]
where $*,\square\in\{\detada, \detnon, \ranada, \rannon\}$
and $C$ is a universal constant.
If $X$ or $Y$ is a Hilbert space or if $F$ is symmetric, the improvements of Theorem~\ref{thm:intro} apply.
\end{coro}

\medskip

This corollary is not as obvious as it seems at first glance.
The problem is that, due to the assumed measurability of randomized algorithms, we do \emph{not} have the relation $\mathcal{A}_n^\detnon \subset \mathcal{A}_n^\rannon$.
What we do have is the relation $\mathcal{A}_n^\detnonmb \subset \mathcal{A}_n^\rannon$, 
where $\mathcal{A}_n^\detnonmb$ denotes the class of all $(\mathcal{B}_F,\mathcal{B}_Y)$-measurable deterministic and non-adaptive algorithms with the corresponding minimal worst-case error denoted by $e_n^\detnonmb$.
The issue is fixed by the following lemma, which shows that 
measurable deterministic and non-adaptive algorithms are (roughly) as good as arbitrary deterministic non-adaptive algorithms.
i.e., there is no real difference between $\mathcal{A}_n^\detnonmb$ and $\mathcal{A}_n^\detnon$.

\begin{lemma}\label{lem:meas}
    Let $X$ and $Y$ be Banach spaces and $S\in\calL(X,Y)$. For every bounded and convex $F\subset X$ and $n\in\N_0$, we have
     \[
     e_n^\detnonmb(S,F) \,\le\, 8\, e_n^\detnon(S,F).
    \]
\end{lemma}

\medskip

Lemma~\ref{lem:meas} is proven in \cite[Thm.~11(v)]{Mathe90} for the case that $F$ is the unit ball of the space $X$. 
(In fact, it is shown that continuous algorithms are almost optimal.) We show below how it can be transferred to general convex classes. It is open
whether the factor 8 can be removed and 
to what extent measureable algorithms are as good as non-measurable algorithms also in other settings,
see \cite[Section~4.3.3]{NW1} and \cite{NR89}.

\medskip

\begin{proof}[Proof of Lemma~\ref{lem:meas}]
Without loss of generality we assume that $0\in F$; the error numbers do not change if we shift $F$. Then we have $F\subset F-F$. The class $F-F$ is convex, bounded and symmetric, and hence the unit ball of a norm 
on $X$. 
By \cite[Thm.~11(v)]{Mathe90}, it holds that
\[
 e_n^\detnonmb(S,F)
\,\le\, e_n^\detnonmb(S,F-F)
\,\le\, 2\,e_n^\detnon(S,F-F).
\]
On the other hand, \cite[Lemma~4.3]{NW1} and Proposition~\ref{prop:upper-c} give
$$e_n^\detnon(S,F-F)
\,\le\, 2\,c_n(S,F-F)
\,=\, 4\,c_n(S,F)
\,\le\, 4\,e_n^\detnon(S,F).$$
\end{proof}

We can now prove Corollary~\ref{coro:individual}.
\medskip

\begin{proof}[Proof of Corollary~\ref{coro:individual}]
All the minimal errors are bounded from below by $\frac12 \cdot h_{2n-1}(S,F)$.
For $\mathcal A_n^\ranada$ and $\mathcal A_n^\rannon$, this follows from Proposition~\ref{prop:lower-h}.
For $\mathcal A_n^\detada$ and $\mathcal A_n^\detnon$, it follows from \cite[Prop.~1]{Novak95b} and Proposition~\ref{prop1}. 
On the other hand, all the minimal errors are bounded from above by $16\cdot c_n(S,F)$.
For $\mathcal A_n^\detada$ and $\mathcal A_n^\detnon$, this follows from Proposition~\ref{prop:upper-c}.
Lemma~\ref{lem:meas} implies that the upper bound also holds for $A_n^\detnonmb$
and thus for $\mathcal A_n^\ranada$ and $\mathcal A_n^\rannon$. 
Hence, the statement follows from Theorem~\ref{thm:widths}.\\
\end{proof}

We 
summarize the state of the art for the maximal gain between the different classes of algorithms, 
see Table~\ref{table:individual}. 
For this, 
let us define the ``maximal gain function''
\[
 {\rm gain}(*,{\square}, \triangle)
 \,:=\, \sup_{\substack{S\in\calL(X,Y)\\ F\subset X \text{ is } \triangle}}
\left({\rm rate}\Big(e_n^{\square}(S,F)\Big) - {\rm rate}\Big(e_n^{*}(S,F)\Big) \right), 
\]
where 
the rate function is defined in~\eqref{eq:rate}, 
$*,\square\in\{\detada$, $\detnon$, $\ranada$, $\rannon\}$
and 
$\triangle\in\{\text{convex}$, $\text{ convex+symmetric}\}$.

\begin{table}[H]
\begin{center}
\begin{tabular}{|c||>{\centering}p{15mm}|>{\centering}p{15mm}|>{\centering}p{15mm}|c|}
\hline
\parbox[0pt][2.2em][c]{2cm}{Gain from } & \multicolumn{2}{|c|}{$\Adetnon$ } & $\Arannon$ & \parbox[0pt][0em][c]{15mm}{\centering $\Adet$} \\ \hline
\backslashbox{for}{to} & $\Adet$ & $\Arannon$ & \multicolumn{2}{|c|}{$\Aran$ }\\\hline\hline
\parbox[0pt][2.2em][c]{4cm}{$F$ convex+symmetric} & 0 &  $\left[\frac12,1\right]$ & 
$ 1 $ &  1 \\\hline
\parbox[0pt][2.2em][c]{2cm}{$F$ convex} & $\left[\frac12,\frac32\right]$ & $\left[\frac12,\frac32\right]$ & $\left[1, 
\frac32\right]$ & $\left[1,\frac32\right]$ \\\hline
\end{tabular}
\end{center}
\caption{Maximal gain in the rate of convergence between different classes of algorithms using linear information.}
\label{table:individual}
\end{table}

The upper bounds in Table~\ref{table:individual} are given by Corollary~\ref{coro:individual} and Proposition~\ref{prop:upper-c}.
We now turn to the lower bounds on the (individual) gain of adaption and randomization. For this, we collect specific examples:

\medskip
\begin{itemize}[leftmargin=5mm]
\item $F$ convex+symmetric:  \\[-2mm]
    \begin{enumerate} 
    \item
    \emph{Power of adaption, deterministic} ($\Adetnon\to\Adet$): \\
    There is no gain in the rate of convergence. As stated in Proposition~\ref{prop:upper-c},
    we have $e_n^\detnon(S,F)\le 2\cdot e_n^\detada(S,F)$ for any $S\in\calL$, 
    see for example \cite{CW04,No96,NW1,TWW88}. 
    An example where adaptive algorithms are slightly better can be found in~\cite{KN90}.
    \medskip

    \item
    \emph{Power of randomization, non-adaptive} ($\Adetnon\to\Arannon$): \\
    Randomization of non-adaptive algorithms can yield a gain of 1/2 
    for certain Sobolev embeddings. 
    The simplest case is from $W^k_2([0,1])$ into $L_\infty([0,1])$, 
    where the optimal rate with deterministic algorithms is $n^{-k+1/2}$
    for $k> 1/2$.
    This is a classical result 
    of \cite{ST67}.
        Using non-adaptive randomized algorithms, one can get the upper bound 
    $n^{-k} \log n$, see 
    \cite{Ma91} and 
    \cite{BKN18,FD08,He92} for further results 
    and extensions.
    
    \medskip
    
    \item
    \emph{Power of adaption, randomized} ($\Arannon\to\Aran$): \\
The paper \cite{KNW24} shows that one may gain a factor 
    $(n/\log n)^{1/2}$
    for the embedding $S\colon \ell_1^m \to \ell_2^m$ and suitable (large) $m$ 
    and 
    in  \cite{KW24} it is proved that one may gain,
    up to logarithmic terms, 
    a factor of polynomial order $n$  
    for the embedding $S\colon \ell_1^m \to \ell_\infty^m$
    with appropriate~$m$.  
    Note that this example (or more precisely, an infinite-dimensional version of it) shows the 
    optimality of Corollary~\ref{coro:rate}
    and the factor $n$ in Theorem~\ref{thm:intro-std}. 
    We do not know if a gain of 1 can 
    also occur in the 
    transitions $\Adetnon\to\Arannon$.
        
    \medskip
    
    \item
    \emph{Power of randomization, adaptive} ($\Adet\to\Aran$): \\ 
If we employ Example (1), as well as $\A_n^\detnonmb\subset\Arannon$ and 
Lemma~\ref{lem:meas}, 
we can take the examples from (3),
to obtain the same gain.  

\end{enumerate}
\medskip

\item $F$ convex:  \\[-2mm] 
    \begin{enumerate} 
\setcounter{enumi}{4}
    \item
    \emph{Power of adaption, deterministic} ($\Adetnon\to\Adet$): \\
    Consider $S={\rm id}\in\calL(\ell_\infty,\ell_\infty)$, 
    i.e., approximation in $\ell_\infty$, for inputs from 
    $$ \hspace{18mm}
    F= \{ x \in \ell_\infty \mid x_i \ge 0, \sum x_i \le 1, x_k \ge x_{2k}, x_k \ge x_{2k+1} \}.
    $$ 
    Then one can prove a lower bound $c (\sqrt{n} \log n)^{-1}$ for non-adaptive 
    algorithms while a simple adaptive algorithm 
    (using ``function values'', i.e., values of coordinates of $x$) gives the upper bound 
    $(n+3)^{-1}$;  see \cite{Novak95a} for details. 
    This shows a gain of 1/2. 
    
    In the case of 
    standard information, i.e., function values, 
 see also Section~\ref{sec:sampling}, 
    there are more 
    extreme examples, where adaption yields a gain 
    up to the order $n$, see again 
    \cite{Novak95a}. 
    
    \medskip
    
    \item
    \emph{Remaining Cases}: \\ 
    Clearly, the examples given in the symmetric case also apply here. 
    This gives the remaining lower bounds of Table~\ref{table:individual}. \\
    We do not know any example of non-symmetric $F$, where the gain of adaptive over non-adaptive randomized algorithms, 
    or of randomized over deterministic algorithms (adaptive or not), is larger than the corresponding gain in the 
    symmetric case. 
    
    \medskip

    \end{enumerate}
\end{itemize}

Let us highlight 
a few open problems indicated by Table~\ref{table:individual} that we assume to be 
of particular interest.
In particular, note that 
it is still possible that the maximal gain of 
adaption and randomization is $1$ also for 
non-symmetric convex sets.

\goodbreak
\begin{OP}
\phantom{h}
\begin{enumerate}
\itemsep=3mm
\item Bounds for individual $n$: \\
Verify whether $e_{2n}^\detnon(S,F)\le C\, n\,
e_n^\ranada(S,F)$ for some $C>0$ and all convex and symmetric $F$. 
See also \cite[Thm.~8.6]{Pietsch-s}.
Does a similar bound hold even for all convex classes $F$?

\item Power of adaption: \\
Is there 
some $S\in\calL$ and convex $F$ such that 
$$e_{2n}^{\detada}(S,F)\le C \, 
n^{-1}\, e_n^{\detnon}(S,F)$$ 
for all $n\in\N$? 

\item Power of randomization: \\
Is there 
some $S\in\calL$ and convex $F$ such that 
$$
e_{2n}^{\rannon}(S,F)\le  C \, n^{-1}\, 
e_n^{\detnon}(S,F)
$$ 
for all $n\in\N$?

\item Power of randomization for symmetric sets: \\
Is there 
some $S\in\calL$ and convex and symmetric $F$ 
such that 
$$
e_{2n}^{\rannon}(S,F)\le  C \, n^{-\alpha}\,
e_n^{\detnon}(S,F)
$$ 
for all $n\in\N$ and some $\alpha>\frac12$? 
\end{enumerate}
\end{OP}


\medskip
\subsection{Other widths} \label{sec:other-width}

Our main emphasis was to compare different classes of algorithms. 
However, the study of widths is clearly not only of interest in IBC. 
These notions are used in many areas, including approximation theory, geometry, and the theory of Banach spaces. 

We indicate how our results apply here, and give further references.

\subsubsection{Kolmogorov widths} \label{sec:dn}

Possibly most prominent among the widths are the 
\emph{Kolmogorov widths} 
\[
d_n(S,F) \,:=\, \inf_{\substack{M \subset Y\\ \dim(M)\le n}}\,
\sup_{f\in F}\, 
\inf_{g\in M}\,  \norm{S(f)-g}, 
\]
%
which describe how well $S(f)$ for $f\in F$ can be approximated by elements from an affine linear subspace. Such a \emph{best approximation} can generally not be found by an algorithm and so, 
$d_n$ is usually not a suitable benchmark for algorithms. 
Still, it is a natural ``geometric'' quantity. 

If $F=B_X$, then it is known that the Kolmogorov numbers $d_n$ are 
\emph{dual} to the Gelfand numbers, 
and that the Hilbert numbers are self-dual, see~\cite{Bauhardt,Pietsch-s,Pie87}.
Similar statements might hold for general $F$, 
and this could be used to extend our results to $d_n$. 
However, one may also repeat the proofs of our results almost verbatim for $d_n$.
By this, we obtain for $S\in\calL(X,Y)$ and convex $F\subset X$ that 
\begin{equation} \label{eq:dn}
d_n(S,F) 
\;\le\; (n+1)^{\alpha}\,\bigg(\prod_{k=0}^n h_k(S,F) \bigg)^{1/(n+1)}
\end{equation}
with $\alpha=1$ if $F$ is symmetric and $\alpha=3/2$ otherwise.
The symmetric case is proven in \cite{U24};
the modifications for the non-symmetric case are analogous.
We omit the details.

Of course, the same bound holds with the larger $b_n$ in place of the smaller $h_n$. 
The ratio of Kolmogorov and Bernstein widths was studied 
at least since the paper \cite{MH63} from Mityagin and Henkin (1963). 
They proved that $d_n(S,F)\le(n+1)^2
\, b_n(S,F)$ for convex and symmetric $F$, 
and conjectured that one has indeed 
$d_n(S,F)\le(n+1)\, b_n(S,F)$. 
See also \cite{Novak95b} for the 
non-symmetric case. 
The above considerations show that this old conjecture is true, at least for regular sequences and up to constants. 
Note again that this was known, see~\cite{Pie07}.

Similar problems appear in geometry, where often 
different notions are used, such as successive radii,
or inner and outer radii. 
In this context, their ratio was mainly 
considered for sets in Hilbert spaces, 
see \cite{Merino13,Merino17,Pe87,Pu79}, 
where one can find further references. 
Results for general norms can be found, e.g., in \cite[Theorem 5.1]{Merino17}. 
These bounds are improved by 
the inequalities above.

\begin{remark} 
For $S$ being the identity on a Hilbert space, 
it is conjectured that the 
regular simplex provides 
the largest gap 
in the non-symmetric case and the regular cube or the regular 
cross-polytope provide 
the largest gap 
in the symmetric case. 
The Kolmogorov and Bernstein widths of these 
sets are completely known \cite{Br05,Pu80}, 
but the proven general bounds are slightly 
weaker. 
\end{remark}

\subsubsection{Linear widths} \label{sec:linear} 
If we replace the requirement that the approximation space is linear by the requirement that the approximation procedure is linear,
we end up with 
the 
\emph{approximation numbers of $S$ on~$F$}, 
or \emph{linear widths of $F$ with respect to $S$}, 
i.e., 
\begin{equation*}
a_n(S,F) \;:=\; 
\inf_{\substack{L_1,\dots, L_n\in X'\\ \phi_0, 
\dots,\phi_n\in Y}} \,
\sup_{f\in F} \,
\norm{S(f) - \phi_0 - \sum_{i=1}^n L_i(f)\, \phi_i}.
\end{equation*}
This corresponds to the minimal worst-case 
error of affine 
algorithms that use at 
most $n$ pieces of linear information.

If $F$ is the 
unit ball of $X$, 
then it is known from 
\cite[Thm.~8.4]{Pietsch-s} 
(based on~\cite{KS71}) that
\begin{equation}  \label{an-versus-cn} 
 a_n(S,B_X) \,\le\, (1+\sqrt{n})\, c_n(S,B_X).
\end{equation} 
Note that $a_n(S,B_X)$ and $c_n(S,B_X)$ are 
equal if 
$X$ is a Hilbert space or if $Y$ has the 
metric extension property, 
see~\cite{Pietsch-s} or~\cite{CW04,Mathe90} 
for extensions.
Together with Theorem~\ref{thm:widths}, 
inequality \eqref{an-versus-cn}
implies 
\begin{equation*}
a_n(S,B_X) 
\;\le\; 2\,(n+1)^{3/2} \,\bigg(\prod_{k=0}^n h_k(S,B_X) \bigg)^{1/(n+1)}.
\end{equation*}
The exponent $3/2$ can be replaced by 
$1$ in the aforementioned cases
or in the case that $Y$ is a Hilbert space since 
then we have the identity $a_n(S,B_X)=d_n(S,B_X)$.
Again, for this symmetric case, 
the result is essentially  known, 
see \cite[6.2.3.14]{Pie07}, 
and we only remove an oversampling constant 
compared to the known estimate.
It is a major open problem, whether 
the exponent $3/2$ can be reduced 
to $1$, 
see \cite[Open Problem~5]{Pie09}. 

This problem is of particular 
interest in the theory of s-numbers, 
where $F$ is assumed to be the unit ball of $X$. 
Note that the approximation numbers form 
the largest scale of s-numbers
while the Hilbert numbers are the 
smallest scale of s-numbers,
see \cite[Ch.~2]{Pie87}.


\medskip

\subsubsection{Non-linear widths}
    In connection with nonlinear approximation, 
    also several types of \emph{non-linear widths} appear in the literature, see e.g.~\cite{CDPW,DHM,DKLT93,Siegel,Ste74,Ste75,Yar}. 
    Let us introduce two of them to illustrate the relation to our setting. 
    First, the \emph{manifold widths} of $F\subset X$ with respect to $S\in\calL(X,Y)$ are defined by 
    \[
    \delta_n(S,F) \,:=\,  
    \inf_{\substack{N\in C(X,{\R^n}) \\ \phi\in C(\R^n,Y)}}\, 
    \sup_{f\in F}\, \bigl\|S(f)-\phi\left(N(f)\right)\bigr\|,
    \]
    where $C(X,Y)$ denotes the class of continuous mappings from $X$ to $Y$. 
    Moreover, the \emph{continuous co-widths} of $F\subset X$ w.r.t.~$S\in\calL(X,Y)$ are 
    \[
    \widetilde{c}_n(S,F) \,:=\,  
    \inf_{N\in C(X,{\R^n})}\, 
    \sup_{\substack{f,g\in F: \\ N(f)=N(g)}}
    \frac12\, \bigl\|S(f)-S(g)\bigr\|. 
    \]
    These numbers 
    correspond (up to a factor 2) to minimal errors 
    for approximating $S$ over~$F$ based 
    on $n$ non-adaptive continuous measurements.
    Comparing these definitions to minimal 
    errors and Gelfand numbers,
    we see that this approach is more general 
    in the sense that the information 
    mapping (or parameter selection map) $N$ 
    is not built from linear functionals,
    but less general in the sense that a discontinuous adaptive choice of the one-dimensional measurements (like, e.g., a bisection method) is not allowed.
    These quantities seem to appear in the 
    literature only in the special 
    case of $S\in\calL(X,X)$ being the identity. 
    We naturally extend the definitions, 
    but only comment on this special case in the following.
    
    First, note that another important (and very early) concept are the Aleksandrov widths, see~\cite{Al56,DKLT93}, which replace $\R^n$ by more general $n$-dimensional \emph{complexes}.
    However, it is shown in~\cite{DKLT93} that all these quantities are equivalent up to constants and oversampling. 

    Second, it is clear from the definitions that these 
    widths are smaller than $e_n^\detnon$ and $c_n$, respectively. 
    Moreover, it is shown in~\cite{DHM} that they are lower-bounded by the Bernstein widths $b_n$. So, we have everything we need to apply our technique: 
    The upper bounds 
    in Theorem~\ref{thm:widths} and 
    in Table~\ref{table3} 
    are also applicable to the maximal gain 
    in the rate of convergence when passing 
    from linear to arbitrary continuous 
    measurement maps $N\colon F \to \R^n$, 
    at least for $S$ being the identity. 
    
    In fact, combining Theorem~\ref{thm:widths}, in the form of Corollary~\ref{cor:Carl}, 
    with \cite[Thm.~3.1]{DHM} (and noting that they use the notation ``$d_n$'' for ``$\delta_n$''), 
    we obtain, e.g., for all convex $F\subset X$ that 
    \begin{equation}
    c_{2n}({\rm id}_X,F) 
    \;\le\; C_\alpha \, n^{-\alpha+3/2}\cdot \sup_{k< n} \, (k+1)^{\alpha}\, \delta_k({\rm id}_X,F),
    \end{equation}
    where $C_\alpha \le 16^{\alpha+1}$. 
    Again, the exponent $3/2$ can be replaced by 1 for symmetric $F$ and,  
    taking Section~\ref{sec:dn} into account, 
    this bound also holds with $d_{2n}$ in place of $c_{2n}$. 
    
    For a very different 
    (and surprising for us) result on adaptive continuous measurements that shows an exponential speed-up, see~\cite{KNU25}.
    
\medskip
\subsection{Sampling numbers and other 
classes of information} \label{sec:sampling}

As already indicated in the introduction, our proof 
technique can also be employed for sampling recovery 
in the uniform norm. 
In fact, our upper bound from Theorem~\ref{thm:intro} 
also holds if the class of deterministic, non-adaptive 
algorithms is further restricted to those using only 
certain restricted information. 

For this, we consider a set of linear functionals $
\Lambda\subset X'$, which we call the 
\emph{admissible information}. 
The $n$-th minimal worst-case 
error for approximating $S$ over $F$ with 
information from $\Lambda$
is defined by 
\[
e_n^\detnon(S,F,\Lambda) 
\;:=\; 
\inf_{\substack{L_1,\dots,
L_n\in\Lambda \\ \phi\co \R^n\to Y}} \,
\sup_{f\in F} \, 
\norm{S(f) - \phi(L_1(f),\dots, L_n(f))}.
\]
Moreover, 
we call $\mathcal N\subset X'$ a \emph{norming set} of 
$X$ if 
\[
\norm{f}=\sup\{\,|\lambda(f)|\co \lambda\in \mathcal N\} 
\qquad \text{ for all } \quad f\in X.
\]
With this, we obtain the following 
generalization of Theorem~\ref{thm:intro}.

\medskip

\begin{theorem}\label{thm:info}
Let $X$ and $Y$ be Banach spaces and $S\in\calL(X,Y)$. 
Moreover, let $\Lambda_S\subset\Lambda\subset X'$, 
where $\Lambda_S$ is of the form 
$\Lambda_S=\{\lambda\circ S\co \lambda\in \mathcal N\}$ 
and $\mathcal N\subset Y'$ is a norming set of $Y$. 
Then, for every convex 
$F\subset X$ and $n\in\N$, we have
\begin{align*}
e_{2n-1}^\detnon(S,F,\Lambda) 
\;&\le\; \left(\prod_{k<2n} e_k^{\detnon}(S,F,\Lambda)\right)^{1/(2n)}\\
&\le\; 12\, n^{3/2}\,
\bigg(\prod_{k<n} e_k^\ranada(S,F) \bigg)^{1/n}.
\end{align*}
The corresponding improvements from 
Theorem~\ref{thm:intro} apply 
if $F$ is additionally symmetric or the unit ball of a 
Hilbert space $X$.
\end{theorem} 

\medskip

\begin{proof}[Proof of Theorem~\ref{thm:info}]
First note that the bound
\begin{equation}\label{eq:UBGelf}
     e_n^{\detnon}(S,F,\Lambda) \,\le\, 2\, c_n(S,F,\Lambda)
\end{equation}
from Proposition~\ref{prop:upper-c} 
also holds 
for the generalized
Gelfand numbers
\[
 c_n(S,F,\Lambda) 
\,:=\,  
\inf_{L_1,\dots,L_{n}\in \Lambda}\, 
\sup_{\substack{f,g\in F: \\ L_k(f)=L_k(g)}}
\frac12\, \bigl\|S(f)-S(g)\bigr\| 
\]
see \cite[Sec.~4.1]{NW1}. 
Hence it suffices to bound the modified Gelfand numbers in terms of the Hilbert numbers
as in Theorem~\ref{thm:widths}.

We proceed as in the proof of Theorem~\ref{thm:widths}.
By the definition of a norming set $\mathcal N\subset Y'$, 
we can choose the functionals $\lambda_k$ in~\eqref{proof1} from $\mathcal N$, 
arguing with a slightly smaller $\eps$ in the previous inequality. 
Hence, the functionals $L_k$ defining $M_k$ are from $\Lambda_S$ and we can choose $f_k, g_k$ such that \eqref{eq:choosep} holds with the modified Gelfand widths. 
This shows that the assertion in the very beginning of the proof of Theorem~\ref{thm:widths} holds with $c_k(S,F)$ replaced by $c_k(S,F,\Lambda)$.

The same replacement is possible for the corresponding assertion in the case that $F$ is the unit ball of a Hilbert space.
We can therefore copy the rest of the proof of Theorem~\ref{thm:widths} with $c_k(S,F)$ replaced by $c_k(S,F,\Lambda)$.\\
\end{proof}

Recall that the minimal errors on the right hand side in Theorem~\ref{thm:info} are for the (much bigger) class of randomized, adaptive algorithms that have access to 
arbitrary linear functionals, see Section~\ref{sec:algorithms}. 
We do not consider the case that $Y$ is a Hilbert space since we believe that norming sets in Hilbert 
spaces are too large to yield an interesting 
generalization of Theorem~\ref{thm:intro}.
Let us also note that 
our proof of Theorem~\ref{thm:intro} only employs information of the form $\Lambda_S$. 
As such, Theorem~\ref{thm:intro} cannot catch the optimal behavior of $e_n^\detnon(S,F)$
if the latter decay faster than $e_{n}^\detnon(S,F,\Lambda_S)$.

\medskip

We discuss a few examples:

\begin{enumerate}[leftmargin=10mm]
\item \emph{Linear information}:\\
The case studied in 
Theorem~\ref{thm:intro} corresponds to 
$\mathcal N=B_{Y'}$
and $\Lambda=X'$
which 
satisfy the assumptions of the theorem. 
The class $B_{Y'}$ is a norming set by the Hahn-Banach theorem.

\medskip

\item \emph{Uniform approximation (and proof of Theorem~\ref{thm:intro-std})}: \label{ex:uniform}\\
We consider $X=Y=B(D)$ and 
$S
= \APP_{\infty}$, 
i.e., the identity on $B(D)$, and  
observe that 
$\Lambda^{\rm std}:=\{\delta_x\co x\in D\}$ with 
the Dirac functionals $\delta_x(f)=f(x)=Sf(x)$ 
is a norming set of $B(D)$. 
Since $g_n(S,F)= e_{n}^\detnon(S,F,\Lambda^{\rm std})$, 
see~\cite{CW04} or~\cite[Thm.~4.8]{NW1}, 
we obtain Theorem~\ref{thm:intro-std} 
from Theorem~\ref{thm:info}.
A special case is the space of bounded sequences 
$\ell_\infty=B(\N)$, 
where $\Lambda^{\rm std}$ consists of evaluations of coordinates.
Note that the factor~12 from Theorem~\ref{thm:info} can be replaced with~6 in Theorem~\ref{thm:intro-std} since the factor 2 in \eqref{eq:UBGelf} can be removed in this case,
see again \cite[Sec.~4.1]{NW1}.

\medskip

\item \emph{$C^k$-approximation}:\\
We can also apply Theorem~\ref{thm:info} to function recovery 
in $C^k(D)$, the space of $k$-times continuously differentiable functions on a compact domain $D\subset \R^d$.
That is, we consider the identity $S$ on the space $X=Y=C^k(D)$,
which we equip with the norm
\[
 \Vert f \Vert_{C^k(D)} \,:=\,
 \max_{|\alpha|\le k}\ \max_{x\in D}\, 
 \left\vert \frac{\partial^{|\alpha|}
 f(x)}{\partial_{x_1}^{\alpha_1} \cdots 
 \partial_{x_d}^{\alpha_d}} \right\vert. 
\]
Theorem~\ref{thm:info} applies 
for the class $\Lambda$ of point evaluations of derivatives up to order $k$, i.e., for
\[
 \Lambda \,=\, \left\{ \delta_x \circ \frac{\partial^{|\alpha|}
 }{\partial_{x_1}^{\alpha_1} \cdots 
 \partial_{x_d}^{\alpha_d}} \ \Big\vert\ x\in D,\ \vert \alpha\vert \le k  \right\},
\]
which is a norming set on $C^k(D)$.


\end{enumerate}

\medskip

We discuss two implications 
of Theorem~\ref{thm:intro-std}.


Firstly, Theorem~\ref{thm:intro-std} contributes to the question 
on the power of adaption and randomization 
if only function values are available.
Namely, considering algorithms for uniform approximation that only use standard information, 
Theorem~\ref{thm:intro-std} gives that adaption and 
randomization cannot lead to a speed-up larger than 
one, i.e.,
\begin{equation*}\label{eq:maxgainstd}
{\rm rate}\Big(e_n^\detnon(\APP_{\infty},F, \Lambda^{\rm std} )\Big)
\;\ge\; 
{\rm rate}\Big(e_n^{\ranada}(\APP_{\infty} ,F,  \Lambda^{\rm std} )\Big) 
\,-\, 1
\end{equation*}
for any convex and symmetric $F\subset B(D)$.
This is in analogy to our result for linear 
information,
see 
Corollary~\ref{coro:rate}.

Secondly, Theorem~\ref{thm:intro-std} also 
contributes to the question on the power of
standard information compared to arbitrary linear information.
If we consider randomized algorithms for uniform approximation, Theorem~\ref{thm:intro-std} gives that a restriction to standard information causes a loss in the rate of convergence of no more than one, i.e.,
\[
{\rm rate}\Big(e_n^{\ranada}(\APP_{\infty} ,F,  \Lambda^{\rm std} )\Big) 
\;\ge\; {\rm rate}\Big(e_n^\ranada(\APP_{\infty},F
)\Big) \,-\, 1
\]
for any convex and symmetric $F\subset B(D)$.
The analogous result for deterministic algorithms has been proven in~\cite{Novak}.
This has recently been improved to
\begin{equation}\label{eq:loss-det}
{\rm rate}\Big(e_n^{\detada}(\APP_{\infty} ,F,  \Lambda^{\rm std} )\Big) 
\;\ge\; {\rm rate}\Big(e_n^\detada(\APP_{\infty},F
)\Big) \,-\, 1/2,
\end{equation}
see \cite{KPUU24}.
Note that we even have 
equality of the rates if $F$ is the unit ball of a certain kind of reproducing kernel 
Hilbert space, see \cite{GW23,KPUU23}.

It is an interesting open problem whether 
\eqref{eq:loss-det} also holds
in the randomized setting, 
and to what extent 
the results above hold for more 
general problems $S\in\mathcal{L}(X,Y)$.

\medskip

\begin{remark}[Sampling numbers in $L_2$]
Another problem where several new bounds have been 
obtained recently, is the 
case that $S=\APP_2$, i.e., the 
embedding of $X$ into 
the space $Y=L_2$. 
In this case, there are various upper bounds 
for the error of non-adaptive 
algorithms based on function values in terms of the Kolmogorov numbers $d_n(S,F)$,  
see \cite{DKU,KU1,KU2,NSU,T20} for deterministic
and \cite{CD,K19,WW07} for randomized algorithms.
On the other hand, 
the bound \eqref{eq:dn} 
and Lemma~\ref{lemma92} 
give an upper bound on $d_n(S,F)$ in terms of the 
error of adaptive randomized algorithms.
Hence, we may derive 
several bounds on the 
maximal gain of adaption and/or 
randomization for the problem of 
sampling recovery in $L_2$.

We only mention the special case that 
$F$ is the unit ball of a reproducing 
kernel Hilbert space $X=H$ with finite trace.
Using that
${\rm rate}(e_n^{\detnon}(\APP_2 ,B_H,  \Lambda^{\rm std} ))={\rm rate}(c_n(\APP_2 ,B_H))$ from \cite[Corollary~1]{KU1} together 
with Lemma~\ref{lemma92} and Proposition~\ref{prop1}, 
we obtain that 
\[
{\rm rate}\Big(e_n^\detnon(\APP_2,B_H, \Lambda^{\rm std} )\Big)
\;=\;{\rm rate}\Big(e_n^{\ranada}(\APP_2 ,B_H )\Big).
\]
That is, linear sampling algorithms are optimal (in the sense of order) 
among arbitrary adaptive, randomized algorithms that may use general linear information.

\end{remark}

\begin{remark}[Exponential decay]
For many classes $F$ of smooth functions, the $n$-th minimal error has a super-polynomial decay
and Theorem~\ref{thm:intro-std} 
together with \eqref{eq:exponential}
implies a bound of the form
\[
g_{cn}^{\rm lin}(\APP_{\infty},F) 
\;\le\; e_n^\ranada(\APP_{\infty},F)
\]
for all $n\ge n_0$,
where $n_0\in\N$ and $c\ge 1$ are (relatively small) constants.
One such example is given by reproducing kernel Hilbert spaces with a Gaussian kernel, see, e.g., \cite[Thm~1.1]{KS24}.
This means that, for all such examples, there is no need for sophisticated algorithms that use randomization, adaption or general linear information, at least from the viewpoint of information complexity. In comparison to deterministic and non-adaptive algorithms that only use function evaluations, at most a factor $c$ can be gained.
\end{remark}

\medskip

\goodbreak

\medskip

\noindent \textbf{Acknowledgement.} \ 
We gratefully acknowledge the support of the Leibniz Center for Informatics,
where several discussions on this research were held during the Dagstuhl Seminar
\emph{Algorithms and Complexity for Continuous Problems} (Seminar ID~23351).
This research was funded in whole or in part by the Austrian Science Fund (FWF) grant M~3212-N. 
For open access purposes, the authors have applied a CC BY public copyright
license to any author-accepted manuscript version
arising from this submission. 
MU is supported by the Austrian Federal Ministry of Education, Science and Research via the Austrian Research Promotion Agency (FFG) through the project FO999921407 (HDcode) funded by the European Union via NextGenerationEU.

\medskip


\bigskip
\noindent
\address{D.K., Faculty of Computer Science and Mathematics, University of Passau, Germany; \texttt{david.krieg@uni-passau.de};  \\
E.N., Institute of Mathematics, Friedrich Schiller University Jena, 
Germany; \texttt{erich.novak@uni-jena.de}; \\
M.U., Institute of Analysis
\& Department of Quantum Information and Computation, 
Johannes Kepler University Linz, Austria;\\ \texttt{mario.ullrich@jku.at}
}


\begin{thebibliography}{77}

\bibitem{Al56}
P. Aleksandrov, {\it Combinatorial Topology}, Vol. 1, Graylock Press, Rochester, NY, 1956.


\bibitem{Ba71}
N. S. Bakhvalov,
{\it On the optimality of linear methods for operator 
approximation in convex classes of functions}, 
USSR Comput. Maths. Math. Phys. 11, 244--249, 1971. 

\bibitem{Ba15}
N. S. Bakhvalov, 
{\it On the approximate calculation of multiple integrals} ,
J.~Complexity 31, 502--516, 2015. 


\bibitem{Bauhardt}
W.~Bauhardt, \textit{Hilbert-Zahlen von Operatoren in Banachr\"aumen}, Math. Nachr. 79, 
181--187, 
1977.









\bibitem{Br05}
R. Brandenberg,
{\it Radii of regular polytopes}, 
Discrete Comput. Geom. 33, 43--55, 2005. 






\bibitem{BKN18}
G. Byrenheid, R. J. Kunsch and Van Kien Nguyen,
{\it Monte Carlo methods for uniform approximation on periodic Sobolev spaces 
with mixed smoothness},
J. Complexity 46, 2018, 90--102. 








\bibitem{CDPW}
A. Cohen, R. DeVore, G. Petrova, P. Wojtaszczyk,
{\it Optimal stable nonlinear approximation},
Found. Comput. Math., 22:607--648, 2022.


\bibitem{CD} A. Cohen, M. Dolbeault, {\it Optimal pointwise sampling for $L_2$ approximation}, Journal of Complexity, 68, 101602, 2022.


\bibitem{CW04} 
J. Creutzig, P. Wojtaszczyk,
{\it Linear vs. nonlinear algorithms for linear problems},
J. Complexity, 20, 807--820, 2004. 


\bibitem{DHM}
R. A. DeVore, R. Howard, and C. Micchelli, 
{\it Optimal nonlinear approximation}, 
Manuscripta Mathematica, 63:469--478, 1989.

\bibitem{DKLT93}
R.A. DeVore, G. Kyriazis, D. Leviatan, V. M. Tikhomirov, 
{\it Wavelet compression and nonlinear $n$-widths}. 
Adv. Comput. Math. 1, 197--214, 1993. 


\bibitem{DKU} M.\,Dolbeault, D.\,Krieg, and M.\,Ullrich,  
{\it A sharp upper bound for sampling numbers in $L_2$}, 
{Appl. Comput. Harmon. Anal.} 63, 113--134, 2023. 





\bibitem{EL13}
D. E. Edmunds and J. Lang,
{\it Gelfand numbers and widths},
J. Approx. Theory 166, 2013, 78--84.

\bibitem{FD08}
G. Fang and L. Duan, 
{\it The complexity of function approximation on Sobolev spaces 
with bounded mixed derivtive by linear Monte Carlo methods}, 
J. Complexity 24, 2008, 398--409. 





\bibitem{GM80}
S. Gal and C. A. Micchelli, 
{\it Optimal sequential and non-sequential 
procedures for evaluating a functional},
Appl. Anal. 10, 105--120, 1980. 




\bibitem{GW23}
J.~Geng and H.~Wang, 
{\it On the power of standard information for tractability for $L_\infty$ approximation of periodic functions in the worst case setting},
J. Complexity 80, 101790, 2024.


\bibitem{KS24}
T. Karvonen and Y. Suzuki, 
{\it Approximation in Hilbert spaces of the Gaussian and related analytic kernels}, 
arXiv:2209.12473.


\bibitem{Merino13}
B. Gonz\'alez Merino, 
{\it On the ratio between successive radii of a symmetric convex body},
Math. Ineq. Appl. 16, 569--576, 2013.

\bibitem{Merino17}
B. Gonz\'alez Merino, 
{\it Improving bounds for 
the Perel'man-Pukhov quotient for inner and 
outer radii},
J. Convex Analysis 24, 1099-1116, 2017. 







\bibitem{He89}
S.~Heinrich,
\textit{On the relation between linear n-widths and approximation numbers}.
J.~Approximation Theory 58, No.3, 315--333, 1989.

\bibitem{He92}
S. Heinrich, 
\textit{Lower bounds for the complexity of Monte Carlo function approximation}, 
J. Complexity 8, 277--300, 1992. 



\bibitem{He24a}
S.~Heinrich, 
\textit{Randomized complexity of parametric integration and the role of adaption {I}. 
{F}inite dimensional case},  
J. Complexity, 
81,
101821,
2024.

\bibitem{He24b} 
S.~Heinrich, 
\textit{Randomized Complexity of Parametric Integration and the Role of Adaption {II}. {S}obolev Spaces},
J. Complexity, 
82,
101823,
2024.

\bibitem{He24c} 
S.~Heinrich, 
\textit{Randomized complexity of vector-valued approximation.}
In: A. Hinrichs, P. Kritzer, F. Pillichshammer (eds), Monte Carlo and Quasi-Monte Carlo Methods. MCQMC 2022. Springer Proceedings in Mathematics \& Statistics, vol. 460.

\bibitem{He24d}
S.~Heinrich, 
\textit{Randomized complexity of mean computation and the 
adaption problem}, 
J. Complexity, 
85, 
101872,
2024.


	









\bibitem{Ism74}
R. S. Ismagilov, 
\textit{Diameters of sets in normed linear spaces and approximation of functions by trigonometric polynomials}, Russian Math. Surveys 29(3), 169--186, 1974. 

\bibitem{KS71}
M. Kadets and S. Snobar, 
\textit{Certain functionals on the 
Minkowski compactum},
Math. Notes 10, 694--696, 1971. 

\bibitem{KN90}
M. A. Kon and E. Novak,
\textit{The adaption problem for approximating linear operators}, 
Bull. of the AMS 23, 1990, 159--165. 










 




\bibitem{K19} D. Krieg, {\it Optimal Monte Carlo methods for $L_2$-approximation}, Constr. Approx., 49, 385--403, 2019.

\bibitem{KNU25} 
D. Krieg, E. Novak and M. Ullrich,
{\it How many continuous measurements are needed to learn a vector?},
arXiv:2412.06468, 2025.

\bibitem{KPUU23}
D.\,Krieg, K.\,Pozharska, M.\,Ullrich, T.\,Ullrich,
{\it Sampling recovery in $L_2$ and other norms},
arXiv:2305.07539, 2023.

\bibitem{KPUU24}
D.\,Krieg, K.\,Pozharska, M.\,Ullrich, T.\,Ullrich,
{\it Sampling projections in the uniform norm},
arXiv:2401.02220, 2024.





\bibitem{KU1} D. Krieg and M. Ullrich,
{\it Function values are enough for $L_2$-approximation}, 
Found. Comp. Math., 21(4), 1141--1151, 2021.

\bibitem{KU2} D. Krieg and M. Ullrich,
{\it Function values are enough for $L_2$-approximation: Part II}, 
J. Complexity 66, 101569, 2021. 



\bibitem{Ku16} R. J. Kunsch,
\textit{Bernstein Numbers and Lower Bounds for the Monte Carlo Error}.
In: R. Cools and D. Nuyens (eds), Monte Carlo and Quasi-Monte Carlo Methods. 
Springer Proceedings in Mathematics \& Statistics, vol. 163.


\bibitem{Ku17}
R. J. Kunsch, 
\textit{High-dimensional function approximation: 
breaking the curse with Monte Carlo methods},
Ph. D. dissertation, arxiv:1704.08213.


\bibitem{KNW24}
R. J. Kunsch, E. Novak and M. Wnuk,
\textit{Randomized approximation of summable sequences -- adaptive and non-adaptive}, 
J. Approx. Theory,
304,
106056,
2024.


\bibitem{KW24}
R. J. Kunsch, M. Wnuk,
\textit{Uniform approximation of vectors using 
adaptive randomized information},
arXiv:2408.01098.















\bibitem{LGM96} 
G. G. Lorentz, M. v. Golitschek and Yu. Makovoz, 
{\it Constructive Approxmation, Advanced Problems}, 
Grundlehren 304, Springer-Verlag Berlin, Heidelberg 1996. 




\bibitem{Mathe90}
P. Math\'e, \textit{s-numbers in Information-Based
Complexity}, J. Complexity 6, 41--66, 1990.


\bibitem{Ma91}
P. Math\'e, \textit{Random approximation of Sobolev embeddings}, 
Complexity 7, 261--281, 1991. 







\bibitem{MH63}
B. S. Mityagin, G. M. Henkin, 
\textit{Inequalities between n-diameters}, in Proc. of the Seminar on Functional Analysis 7,
Voronezh, 97--103, 1963.

\bibitem{NSU} N. Nagel, M. Sch\"afer, and T. Ullrich,
{\it A new upper bound for sampling numbers}, Found. Comp. Math. 22(2), 445-468, 2022.






\bibitem{Novak} E. Novak, {\it Deterministic and stochastic error bounds in numerical analysis},
Lecture Notes in Mathematics 1349, Springer-Verlag, 1988.


\bibitem{Novak92} E. Novak. {\it Optimal linear randomized methods for linear operators in Hilbert
spaces}, J. Complexity, 8(1):22--36, 1992.

\bibitem{Novak95a}
E. Novak, 
\textit{Optimal recovery and $n$-widths for convex classes of functions},  
J. Approx. Theory 80, 390--408, 1995.


\bibitem{Novak95b}
E. Novak, 
\textit{The adaption problem for nonsymmetric convex sets}, 
J. Approx. Theory 82, 123--134, 1995.

\bibitem{No96}
E. Novak, 
\textit{On the power of adaption},  
J. Complexity 12, 199--237, 1996. 

\bibitem{NR89}
E. Novak and K. Ritter, 
\textit{A stochastic analog to Chebyshev centers and optimal average case algorithms}, J. Complexity 5, 60--79, 1989.




\bibitem{NW1}
E.~Novak and H.~Wo\'zniakowski,
\newblock {\em Tractability of multivariate problems. {V}olume~{I}: {L}inear
information}, volume~6 of {\em EMS Tracts in Mathematics}.
\newblock European Mathematical Society (EMS), Z\"urich, 2008.






\bibitem{Pe87} 
G. Ya. Perel'man,
{\it On the $k$-radii of a convex body}, 
Siberian Math. J. 28, 665--666, 1987. 




\bibitem{Pietsch-s}
A. Pietsch, {\it s-numbers of operators in Banach spaces}, 
Studia Math. 51, 201--223, 1974.

\bibitem{Pietsch-ideals} A. Pietsch. {\it Operator ideals}, 
North-Holland Mathematical Library 20, Elsevier, 1980.


\bibitem{Pie87}
A. Pietsch. 
{\it Eigenvalues and s-numbers}, 
Cambridge studies in advanced mathematics 13, Cambridge University Press, 1987.

\bibitem{Pie07} 
A. Pietsch, {\it History of Banach spaces and linear operators}, Birkhäuser Boston, MA, 2007.


\bibitem{Pie09}
A. Pietsch, \textit{Long-standing open problems of Banach space
theory: My personal top ten,} Quaestiones Mathematicae 32:3, 321--337, 2009. 

\bibitem{Pinkus85}
A. Pinkus, {\it n-widths in approximation theory}, 
Ergebnisse der Mathematik und ihrer Grenzgebiete. 
Springer Berlin, Heidelberg, 1985. 


\bibitem{Pu79}
S. V. Pukhov, 
{\it Inequalities for the Kolmogorov and Bernstein widths 
in Hilbert space},
Math. Notes 25, 320--326, 1979. 
[Translated from Matematicheskie Zametki, Vol. 25, No. 4, pp. 619–628, April, 1979.] 

\bibitem{Pu80}
S. V. Pukhov,
{\it Kolmogorov diameters of a regular simplex}, 
Moscow Univ. Math. Bull. 35, 38--41, 1980.





\bibitem{Siegel}
J. W. Siegel,
{\it Sharp lower bounds on the manifold widths of Sobolev and Besov spaces},
J. Complexity,
85,
101884,
2024.






\bibitem{ST67}
M. E. Solomjak and V. M. Tichomirov,
{\it Some geometric characteristics of the embedding map 
from $W^a_p$ to $C$}, 
Izv. Vys\v s. U\v cebn. Zaved Mat. 10, 76--82, 
1967. 



\bibitem{Ste74}
M. I. Stesin, 
{\it On Aleksandrov diameters of balls}, 
Dokl. Akad. Nauk SSSR 217, 1, 31--33, 1974. 


\bibitem{Ste75}
M. I. Stesin, 
{\it Aleksandrov diameters of finite-dimensional 
sets and classes of 
smooth functions}, 
Dokl. Akad. Nauk SSSR 220, 6, 1278--1281, 1975. 










\bibitem{T20} V.\,N.\,Temlyakov, {\it On optimal recovery in $L_2$},  
J. Complexity, 65, 101545, 2021. 






\bibitem{Tikh60} 
V. M. Tikhomirov, \textit{Diameters of sets in
function spaces and the theory of best approximations}, 
Russ. Math. Survey 15(3), 75--111, 1960. 

\bibitem{TW80}
J.~Traub and H.~Wo\'zniakowski, 
\emph{A General Theory of Optimal Algorithms},
Academic Press, 1980.


\bibitem{TWW88}
J.~Traub, G.~Wasilkowski, and H.~Wo\'zniakowski, \emph{Information-Based Complexity}, 
Acad.~Press, 1988.



\bibitem{U24} M. Ullrich, {\it On inequalities between s-numbers}, 
Adv. Oper. Theory 9, 82, 2024.






\bibitem{WW07} G.W. Wasilkowski and H. Wo\'zniakowski, 
{\it The power of standard information for multivariate approximation 
in the randomized setting}, Mathematics of Computation, 76, 965-988, 2007.



\bibitem{Yar} D. Yarotsky,
{\it Error bounds for approximations with deep ReLU networks},
Neural Networks, 94:103--114, 2017.

\end{thebibliography}
\end{document}